\pdfoutput=1
\documentclass[11pt]{amsart}
\usepackage{amsmath,amsthm,verbatim,amssymb,amsfonts,amscd,graphicx}
\usepackage{graphics}
\usepackage{tensor}
\usepackage{enumitem}
\usepackage[utf8]{inputenc}
\usepackage{graphicx}
\usepackage{xcolor}
\usepackage{hyperref}
\usepackage{float}
\usepackage{harpoon}
\definecolor{brass}{rgb}{0.71, 0.65, 0.36}

\usepackage{tikz-cd}
\usepackage{slashed}

\theoremstyle{plain}

\newtheorem{theorem}{Theorem}[section]

\newtheorem{lemma}[theorem]{Lemma}
\newtheorem{remark}[theorem]{Remark}

\newtheorem{proposition}[theorem]{Proposition}

\theoremstyle{definition}
\newtheorem{definition}[theorem]{Definition}
\theoremstyle{definition}

\def\Ric{\operatorname{Ric}}

\def\sup{\operatorname{sup}}

\def\sup{\operatorname{sup}}

\def\div{\operatorname{div}}

\def\rr{\mathbb{R}}

\numberwithin{equation}{section}

\numberwithin{equation}{section}

\allowdisplaybreaks

\begin{document}

\title[Stability of the Positive Mass Theorem]{Stability of the Positive Mass Theorem Under Ricci Curvature Lower Bounds}

\author[Kazaras]{Demetre P. Kazaras}
\address{Department of Mathematics\\
Duke University\\
Durham, NC 27708, USA}
\email{demetre.kazaras@duke.edu}

\author[Khuri]{Marcus A. Khuri}
\address{Department of Mathematics\\
Stony Brook University\\
Stony Brook, NY 11794, USA}
\email{khuri@math.sunysb.edu}

\author[Lee]{Dan A. Lee}
\address{Graduate Center and Queens College\\
City University of New York\\
365 Fifth Avenue\\
New York, NY 10016, USA}
\email{dan.lee@qc.cuny.edu}


\thanks{M. Khuri acknowledges the support of NSF Grants DMS-1708798, DMS-2104229, and Simons Foundation Fellowship 681443.}

\begin{abstract}
We establish Gromov-Hausdorff stability of the Riemannian positive mass theorem under the assumption of a Ricci curvature lower bound. More precisely, consider a class of orientable complete uniformly asymptotically flat  Riemannian 3-manifolds with nonnegative scalar curvature, vanishing second homology, and a uniform lower bound on Ricci curvature. We prove that if a sequence of such manifolds has ADM mass approaching zero, then it must converge to Euclidean 3-space in the pointed Gromov-Hausdorff sense. In particular, this confirms Huisken and Ilmanen's conjecture on stability of the positive mass theorem under the assumptions described above.  The proof is based on the harmonic level set approach to proving the positive mass theorem, combined with techniques used in the proof of Cheeger and Colding's almost splitting theorem. Furthermore, we show that the same results hold under a more general lower bound on scalar curvature.
\end{abstract}

\maketitle

\section{Introduction}
\label{sec1}
\setcounter{equation}{0}
\setcounter{section}{1}

The Riemannian positive mass theorem states that the ADM mass of a complete asymptotically flat Riemannian manifold
with nonnegative scalar curvature is nonnegative. Moreover, there is a rigidity statement: the mass of such a manifold is zero if and only if it is isometric to Euclidean space. A natural question to then ask is whether small mass implies that the manifold must be close to Euclidean space in some sense. The \textit{positive mass stability conjecture} asserts an affirmative answer to this question, although the precise statement depends on the topology used to define the closeness. The conjecture was first stated 20 years ago in terms of Gromov-Hausdorff distance by Huisken-Ilmanen in \cite[page 430]{HI}, and more recently in terms of intrinsic flat distance by the third author and Sormani in \cite[page 216]{LeeSormani} (see also \cite[Conjecture 10.1]{Sormani}).

\begin{definition} \label{def:AF}
Given $b>0$ and $\tau>\tfrac{1}{2}$, we say that a smooth connected 3-dimensional Riemannian manifold $(M, g)$ is \emph{$(b,\tau)$-asymptotically flat} if there is a compact set $\mathcal{K}\subset M$ such that its complement is diffeomorphic to the complement of the closed unit ball in Euclidean space, and in the coordinates given by this diffeomorphism $\Phi:M\setminus\mathcal{K}\rightarrow  \mathbb{R}^3 \setminus \bar{B}_1^{\mathbb{E}}(0)$, the metric satisfies
\begin{equation}\label{asymptoticallf}
|\partial^\beta \left(g_{ij}-\delta_{ij}\right)(x)|\leq b|x|^{-\tau-|\beta|},
\end{equation}
for all multi-indices $\beta$ with $|\beta|=0,1,2$.
Furthermore, we require the scalar curvature $R_g$ to be integrable over $M$. Under these assumptions, the \emph{ADM mass} is well-defined and given by
\begin{equation}
m(g)=\lim_{r\rightarrow\infty}\frac{1}{16\pi}\int_{S_{r}}\sum_{i,j=1}^3 (g_{ij,i}-g_{ii,j})\nu^j dA,
\end{equation}
where $\nu$ is the Euclidean unit outer normal to the coordinate sphere $S_r:=\partial B^{\mathbb{E}}_r(0)$, and $dA$ represents its area element induced by the Euclidean metric.
\end{definition}

We will refer to $\Phi$ as the \emph{asymptotic coordinate chart} for $(M,g)$, and it is always implicitly part of the data when we refer to a $(b,\tau)$-asymptotically flat manifold $(M, g)$.

When $(M, g)$ is complete and $R_g \geq 0$, the positive mass theorem states that $m(g)\geq 0$, and $m(g)=0$ if and only if $(M,g)$ is isomorphic to Euclidean space. This theorem was established in the late 1970's by Schoen and Yau \cite{SchoenYau,SchoenYau3}, using an indirect proof based on the existence of stable minimal surfaces, along with manipulations of the stability inequality. Witten \cite{PT1,Witten1} later gave an alternative proof in which the mass is directly expressed as the integral of a nonnegative quantity depending on an asymptotically constant harmonic spinor. This argument relies upon the existence of harmonic spinors and the Lichnerowicz formula. Later, Lohkamp \cite{Lohkamp0} explained how the positive mass theorem can be reduced to the nonexistence of positive scalar curvature metrics on the connected sum $N\# T$ of a compact manifold $N$ with a torus $T$ (which is known from~\cite{GromovLawson,SchoenYau79structure}).
In 2001, as a byproduct of their proof of the Penrose inequality, Huisken and Ilmanen \cite{HI} proved the result using a weak version of inverse mean curvature flow and monotonicity of Hawking mass. More recently, in 2018 Yu Li \cite{Li0} gave a proof using Ricci flow.

As for the positive mass theorem in higher dimensions, Witten's proof works for all spin manifolds, while Schoen and Yau were able to iterate their argument up to dimensions less than 8 \cite{SchoenYau79structure}. For dimensions 8 and higher, see the articles by Schoen and Yau \cite{SchoenYau4} and Lohkamp \cite{Lohkamp1}. A survey of many of the results described above may be found in \cite{Lee}.

Most recently, Bray, Stern, and the first two authors \cite{BKKS} obtained a new proof of the positive mass theorem,  which was motivated by Stern's original work in \cite{Stern1}. Similar to Witten's approach, it gives an explicit lower bound for the mass as an integral of a nonnegative quantity depending on an \emph{asymptotically linear harmonic function}.
It is this approach to mass that we have found to be effective for examining the stability question. Since this approach is only known to work in three dimensions, we will restrict our attention to three dimensions throughout the rest of this article.

The question of stability is made delicate by the presence of ``gravity wells'' and
horizons. It is well-known that the mass cannot ``see behind" compact minimal surfaces, and thus these horizons may hide geometry and topology, preventing proximity to Euclidean space. This problem may be resolved, however, by cutting along the outermost minimal surface and discarding the trapped region behind it. For this reason, many results dealing with stability are stated or conjectured for the exterior regions lying outside the outermost minimal surface. The issue of ``gravity wells,'' on the other hand, is not so easily overcome. This imprecise concept refers to long, thin finger-like protrusions \cite[Figure 1]{LeeSormani}, or spikes, that do not violate nonnegative scalar curvature while ``contributing" very little to the mass:
An example of Ilmanen \cite[Figure 1]{SormaniWenger} shows that it is possible to have a sequence of manifolds with masses tending to zero, such that neither the number of gravity wells nor their volume or diameter can be controlled. Moreover, this sequence does not converge in the pointed Gromov-Hausdorff sense to Euclidean space. One proposal to deal with this problem is to somehow cut out the wells to obtain a modified exterior region which converges properly, as described in the stability conjecture of Huisken-Ilmanen~\cite[page 430]{HI}. An alternative approach is to use a different form of convergence which does not require removal of the gravity wells,
 such as the intrinsic flat distance of Sormani-Wenger \cite{SormaniWenger}, whose design was motivated in part by the phenomenon exhibited by the Ilmanen examples.

Previous results concerning positive mass stability have primarily focused on
establishing intrinsic flat convergence. Several special cases have been treated so far. This includes the third author and Sormani's stability theorem under the assumption of spherical symmetry \cite{LeeSormani}, with its extensions to the asymptotically hyperbolic realm by Sakovich-Sormani \cite{SakovichSormani} as well as to the spacetime regime in the work of Bryden, the second author, and Sormani~\cite{BrydenKhuriSormani}. Various papers by Huang, the third author, Sormani, Allen, and Perales \cite{ AllenPerales,HuangLee,HuangLeePerales,HuangLeeSormani} treated the graphical setting, and Sormani-Stavrov \cite{SormaniStavrov} studied the case of geometrostatic manifolds. Furthermore, results that do not quite establish intrinsic flat convergence but do obtain Sobolev bounds on the metric may be found in \cite{Allen,Bryden}; see also the earlier work surveyed in~\cite{LeeSormani}.

In the current paper we consider the question of stability in the presence of a lower bound on Ricci curvature. Under such a bound, one does not expect Ilmanen-type ``gravity wells'' to arise, and thus it is reasonable that we are able to obtain convergence in the Gromov-Hausdorff sense. In the following statement of our main theorem, $B_r(p)$ denotes the geodesic ball of radius $r$ centered at a point $p\in M$.

\begin{theorem}\label{t:main}
Fix  $b>0$, $\tau>1/2$,  $\kappa>0$, and a point $\mathbf{p}\in\mathbb{R}^3\setminus \bar{B}_1^{\mathbb{E}}(0)$. Let $\{(M_l,g_l)\}_{l=1}^\infty$ be a sequence of orientable complete $(b,\tau)$-asymptotically flat Riemannian 3-manifolds with points $p_l\in M_l$ satisfying
\begin{enumerate}
\item $\Phi_l(p) =\mathbf{p}$ where $\Phi_l$ is the asymptotic coordinate chart of $(M_l, g_l)$;
\item trivial second homology, $H_2(M_l,\mathbb{Z})=0$; \label{homology_assumption}
\item nonnegative scalar curvature, $R_{g_l}\geq 0$; \label{scalar_assumption}
\item a uniform lower bound on Ricci curvature, $\mathrm{Ric}(g_l)\geq -2\kappa g_l$.
\end{enumerate}
If the ADM masses of $(M_l,g_l)$ converge to zero, then $(M_l,g_l,p_l)$ converges to Euclidean space
 in the pointed Gromov-Hausdorff sense.
\end{theorem}

This result confirms the conjecture of Huisken-Ilmanen \cite[page 430]{HI} under the assumption of uniform asymptotics, trivial second homology, and a uniform lower bound on Ricci curvature, where the removal sets $Z_l$ in their statement are taken to be empty. There is a physical meaning to the Ricci curvature bound in condition $(4)$. The Ricci curvature of a time-symmetric hypersurface in a vacuum spacetime is, up to a sign, the so-called electric part of the spacetime's Weyl curvature, see for instance \cite[Section 3.2]{Jez}. One can interpret, therefore, a Ricci curvature lower bound on $(M,g)$ as a bound on the tidal forces arising from gravitational radiation in the spacetime development of this initial data. The topological hypothesis of trivial second homology is included for technical reasons, namely to allow the applicability of the mass inequality from~\cite{BKKS}, which can run into difficulty when non-separating spheres are present. One noteworthy context where this hypothesis is satisfied is under the ``no horizons" condition which demands that $M_l$ contain no smooth compact minimal surfaces. It should be noted that nonnegative scalar curvature and vanishing second Betti number imply that topologically each $M_l \cong S^3/\Gamma_1 \# \cdots \# S^3/\Gamma_k \# \mathbb{R}^3$ \cite[Corollary 6]{Haslhofer} is a finite connected sum of spherical space forms, with a point removed. In particular, Theorem \ref{t:main} does allows for some degree of nontrivial topology. Furthermore, although in general Gromov-Hausdorff convergence does not imply intrinsic flat convergence, Honda \cite{Honda} and Matveev-Portegies \cite{MatveevPortegies} have shown that this is true in a wide variety of settings, if one assumes a Ricci lower bound. These results together with Theorem \ref{t:main} suggest that $(M_l,g_l, p_l)$ also converges in the pointed intrinsic flat sense to Euclidean space, but we do not pursue this question further.

The main ingredients used to prove  Theorem \ref{t:main}  include the level set proof of the positive mass theorem \cite{BKKS}, and the methods used to prove the Cheeger-Colding almost splitting theorem \cite{CC}. Recall that the almost splitting theorem states that if the Ricci curvature of a manifold is ``almost" nonnegative and there is ``almost" a line, then the manifold ``almost" splits. Thus it can be thought of as a quantitative version of the Cheeger-Gromoll splitting theorem~\cite{CheegerGromoll}. The proof applies the almost Ricci nonnegativity and almost line assumptions to obtain the Abresch-Gromoll excess estimate, from which one shows that a class of harmonic functions must have small Hessian in the  $L^2$ sense. This $L^2$ Hessian estimate then implies that the manifold almost splits in the direction orthogonal to the harmonic level sets. In the current setting, we may view the assumptions of almost nonnegative Ricci and existence of an almost line as being replaced by the mass inequality from~\cite{BKKS}. This inequality, when the mass is small and scalar curvature is nonnegative, essentially gives the desired $L^2$ Hessian control of any asymptotically linear harmonic function. By applying the almost splitting argument in three coordinate directions, we obtain convergence to  Euclidean space.
Throughout this process, we require a Ricci lower bound to retain applicability of the Bishop-Gromov inequality,
 the Cheeger-Colding segment inequality \cite{CC}, and  the Cheng-Yau gradient estimate \cite{ChengYau}, which we review in the next section.

Theorem \ref{t:main} may be generalized by relaxing the nonnegative scalar curvature assumption. We relax the assumption in two different ways. First, we can allow the lower bound on scalar curvature to be negative as long as it has a particular divergence structure, and second, we only need the desired lower bound to ``almost'' hold.

\begin{theorem}\label{t:main1}
Theorem \ref{t:main} remains true if the nonnegative scalar curvature assumption~\eqref{scalar_assumption} is replaced by the following weaker assumption: There exist vector fields $X_l \in C^1_{-2-\delta_l}(M_l)$, for some $\delta_l>0$, and nonnegative functions $\psi_l \in L^1(M_l)$ such that
\begin{equation}\label{alaifiaw3}
R_{g_l}\geq |X_l|_{g_l}^2 +\div_{g_l} X_l -\psi_l,
\end{equation}
where $\psi_l$ satisfies the following two conditions:
\begin{itemize}
\item
$
\| \psi_l \|_{L^{1}(M_l)}\rightarrow 0$, and
\item $\psi_{l}$ vanishes on $\Phi_l^{-1}(\mathbb{R}^3\setminus \bar{B}^{\mathbb{E}}_{\hat{r}}(0))$, for some uniform $\hat{r}$ independent of~$l$.
\end{itemize}

\noindent Furthermore, Theorem \ref{t:main} remains true if the vanishing second homology assumption~\eqref{homology_assumption} is replaced by the assumption that $R_{g_l}$ is bounded above by some constant independent of $l$.
\end{theorem}

Here, $C^1_{-2-\delta_l}$ refers to a weighted decay space, see footnote~\ref{foot}. Note that the assumptions on $\psi_l$ say that the inequality $R_{g_l}\geq |X_l|_{g_l}^2 +\div_{g_l} X_l$ holds modulo a negative term that is small in the $L^1$ sense, and is also uniformly compactly supported. This inequality without the $\psi_l$ term  is reminiscent of the form that the scalar curvature takes in the Jang deformation procedure found in Schoen and Yau's proof of the positive energy theorem for initial data sets~\cite{SchoenYauII}. In that proof, the deformed Jang metrics may be conformally transformed to zero scalar curvature. Following this approach, one may try to prove Gromov-Hausdorff stability by obtaining control on the conformal factors, but the required estimates remain elusive. Instead, we again use the mass inequality \eqref{t:massest} and show that $L^2$ smallness of the Hessian is still implied by the relaxed hypotheses. We point out that the lack of a nonnegative scalar curvature hypothesis in Theorem \ref{t:main1} seems to be a novel feature among the known positive mass stability results.

This paper is organized as follows: The next section is dedicated to recalling background results for metrics with Ricci lower bounds, as well as the mass inequality for asymptotically linear harmonic functions. In Section~\ref{sec3} we prove a uniform supremum bound for these functions over an annular region, which then leads us to uniform estimates for their gradients, and finally, uniform smallness of their $L^2$ Hessians in terms of the mass. In Section \ref{sec4}, we use the Cheeger-Colding segment inequality together with the estimates from Section \ref{sec3} to prove that the Hessians are controlled along certain geodesics which mimic the ones used in the proof of the Cheeger-Colding almost splitting theorem. In Section \ref{sec5} these results are applied to obtain an almost Pythagorean identity, while in Section \ref{sec6}, we establish Gromov-Hausdorff approximations and prove Theorem \ref{t:main}. Finally, we prove Theorem \ref{t:main1} in Section~\ref{sec7}.


\section{Background}
\label{sec2} \setcounter{equation}{0}
\setcounter{section}{2}

Throughout this article, constants depending on only parameters $a_1,a_2,\ldots$ will be denoted by $C(a_1,a_2,\ldots)$. Recall the following standard facts about Riemannian manifolds with Ricci curvature bounded below.
\begin{theorem}[Bishop-Gromov relative volume comparison] \label{thm:BG}
Let $(M, g)$ be a complete Riemannian manifold with $\mathrm{Ric}(g)\geq -2\kappa g$, where $\kappa>0$. Let $q\in M$, and let $0<r<s$. Then
\begin{equation}\label{BiGr}
\frac{|B_r(q)|}{|B^{-\kappa}_{r}|}\geq \frac{|B_s(q)|}{|B^{-\kappa}_{s}|},
\end{equation}
where $|B^{-\kappa}_{r}|$ denotes the volume of a geodesic ball of radius $r$ in the hyperbolic $3$-space with constant curvature $-\kappa$.
\end{theorem}

See, for example, \cite[Chapter 9.1]{Petersen} for a proof.
Our arguments in later sections depend heavily on the following result of Cheeger and Colding~\cite{CC}, which is used to transform integral estimates into estimates that hold along certain minimizing geodesics.

\begin{theorem}[Cheeger-Colding segment inequality]\label{segmentlemma}
Let $(M, g)$ be a complete Riemannian manifold with $\mathrm{Ric}(g)\geq -2\kappa g$, where $\kappa>0$.
For $r>0$, let $f:B_{2r}(p)\rightarrow [0,\infty)$, and consider domains $\Omega_1,\Omega_2\subset B_{r}(p)$ with $\xi_1\in\Omega_1$ and $\xi_2\in\Omega_2$. If we define
\begin{equation}
{\mathcal F}_{f} (\xi_1,\xi_2) := \sup_{\gamma} \int_0^{d(\xi_1,\xi_2)} f(\gamma(s))\,ds,
\end{equation}
where the supremum is taken over minimizing unit speed geodesics $\gamma$ joining $\xi_1$ to $\xi_2$, then there exists a constant $C_{\mathrm{s}}(r,\kappa)$ such that
\begin{equation}
\label{abc7}
\int_{\Omega_1\times \Omega_2} {\mathcal F}_{f} (\xi_1, \xi_2) \,d\xi_1\, d\xi_2 \le
C_{\mathrm{s}}(r,\kappa)\left (|\Omega_1|
+|\Omega_2|\right )
\int_{B_{2r}(p)} f\, dV,
\end{equation}
where $d\xi_1\,d\xi_2$ is the Riemannian product measure on $\Omega_1\times \Omega_2$ and $dV$  denotes Riemannian volume measure on $M$.
\end{theorem}

To see how this leads to estimates along particular geodesics, we consider the following easy \emph{mean value-type inequality}: for any nonnegative $L^1$ function $F$ on a domain $\Omega$, there must exist a point $x^*\in\Omega$ such that $F(x^*)$ is less than or equal to twice the average value over $\Omega$, or
\begin{equation}\label{MV}
F(x^*) \le \frac{2}{|\Omega|} \int_\Omega F\,dV.
\end{equation}
The Markov inequality puts a positive lower bound on the measure of the set of such $x^*$, but we will not need this. If we combine this with the segment inequality, it tells us that there exist $x_1^*\in \Omega_1$ and  $x_2^*\in \Omega_2$ such that any minimizing geodesic $\gamma$ from $x_1^*$ to $x_2^*$ satisfies
\begin{equation}
 \int_0^{d(x_1^*,x_2^*)} f(\gamma(s))\,ds
 \le  2C_{\mathrm{s}}(r,\kappa)\left (\frac{1}{|\Omega_1|}
+\frac{1}{|\Omega_2|}\right )
\int_{B_{2r}(p)} f\, dV.
\end{equation}
Given an integral bound on $f$ and arbitrary points $x_1, x_2\in B_r(p)$, we may not be able to control the integral of $f$ along the minimizing geodesics from $x_1$ to $x_2$, but if we chose $\Omega_1$, $\Omega_2$ to be small balls, then the segment inequality allows us to find \emph{nearby} points $x_1^*$, $x_2^*$ for which we do have that control. Of course, if $\Omega_1$, $\Omega_2$ are small balls, then the terms $\frac{1}{|\Omega_1|}$ and
$\frac{1}{|\Omega_2|}$ become large, but this issue can be mitigated by Bishop-Gromov relative volume comparison. This is roughly the philosophy behind the use of the segment inequality. As we will see, the actual application in this paper is a bit more complicated.

The following is a direct consequence of the Cheng-Yau gradient estimate \cite[Corollary 3.2]{SYRedB} for harmonic functions.
\begin{theorem}[Cheng-Yau gradient estimate]\label{t:CY}
Let $(M, g)$ be a complete Riemannian manifold with $\mathrm{Ric}(g)\geq -2\kappa g$, where $\kappa>0$. Let $q\in M$, and let $0<r<s$. Then there exists a constant $C(r,s,\kappa)$ such that
\begin{equation}\label{eqA.2}
\sup_{B_{r}(q)}|\nabla u|\leq C(r,s,\kappa)\sup_{B_s(q)}|u|,
\end{equation}
for any $g$-harmonic function $u$ on $B_s(q)$.
\end{theorem}

The main recent innovation used in our proof of Theorem~\ref{t:main} is the mass inequality due to Bray, Stern, and the first two authors \cite[Theorem 1.2]{BKKS}.
Before stating this result, we briefly introduce the concept of asymptotically linear harmonic functions on a complete $(b,\tau)$-asymptotically flat manifold $(M,g)$. For $i=1,2,3$, let $x^i$ be the $i$-th component of the asymptotically flat coordinate chart $\Phi$ described in Definition~\ref{def:AF}, and then arbitrarily extend it to a smooth function defined on all of~$M$. Asymptotic flatness implies that $\Delta_g x^i$ lies in the weighted\footnote{\label{foot} See, for example,~\cite[Appendix A]{Lee} for a definition of weighted H\"{o}lder spaces. The space $C^{k,\alpha}_s$ roughly means functions that are $O(|x|^{s})$ and whose derivatives decay one order faster for each derivative up to ``$k+\alpha$" derivatives.} H\"{o}lder space $C^{0,\alpha}_{-1-\tau}(M)$, for any
$\alpha\in(0,1)$. It is well-known that
for $\tau\in(\tfrac{1}{2},1)$, the map
\begin{equation}
\Delta_g : C^{2,\alpha}_{1-\tau}(M)\longrightarrow C^{0,\alpha}_{-1-\tau}(M)
\end{equation}
is surjective, with kernel consisting of constant functions; for example, see~\cite[Theorem A.40]{Lee}. Therefore there exists $v\in C^{2,\alpha}_{1-\tau}(M)$ such that $\Delta_g v = -\Delta_g x^i$. Setting $u^i:= x^i +v$, we see that there exists $u^i$ such that
\begin{equation}\label{hasym}
\Delta_g {u}^i=0,\quad\quad\quad
{u}^i- x^i \in C^{2,\alpha}_{1-\tau}(M).
\end{equation}
We will refer to $u^i$ as an \emph{asymptotically linear harmonic function}, asymptotic to the coordinate function $x^i$. Observe for each $i$, $u^i$ is uniquely defined up to the addition of constants. The functions $u^1, u^2, u^3$ are sometimes called \emph{harmonic coordinates}, though it is important to note that they need only form a coordinate system outside some large compact set.

\begin{theorem}[Mass inequality for asymptotically linear harmonic functions~\cite{BKKS}] \label{t:mass}
Let $(M,g)$ be an orientable complete asymptotically flat manifold with trivial second homology, that is, $H_2(M,\mathbb{Z})=0$, and let $u$ be an asymptotically linear harmonic function as described above. Then the ADM mass satisfies
\begin{equation}\label{t:massest}
m(g) \geq \frac{1}{16\pi} \int_{M}\left(\frac{|\nabla^2 u|^2}{|\nabla u|}+R_g |\nabla u|\right) dV,
\end{equation}
where $\nabla^2 u$ denotes the Hessian of $u$.
\end{theorem}
Note that in the statement of this result in \cite{BKKS}, the integral is taken over an ``exterior region." However, it is clear from the proof that the assumptions of orientability and vanishing second homology allow the integral to be taken over all of $M$.

Observe that if the scalar curvature is nonnegative then the mass would exert $L^2$ control on the Hessian of $u$, \emph{if} we had an upper bound on the gradient of $u$, and then once we have $L^2$ smallness of the Hessian of all asymptotically linear harmonic  functions, we can use Cheeger-Colding arguments to show that our space is ``almost'' Euclidean. In the next section we will use uniform asymptotics, the mass inequality, and the Cheng-Yau gradient estimate to prove a uniform gradient bound for $u$.

\section{Estimates for Asymptotically Linear Harmonic Functions}
\label{sec3} \setcounter{equation}{0}
\setcounter{section}{3}

\textbf{Notation:} For $r>0$, let $S_r:=\partial B^{\mathbb{E}}_r(0)$ denote the standard coordinate sphere of radius $r$,
and for any asymptotically flat $(M, g)$ with asymptotic coordinate chart $\Phi$, if $r>1$, define $\mathcal{S}_r:=\Phi^{-1}(S_r)$ and $M_r$ to be the bounded component of $M\setminus \mathcal{S}_r$.  Or in other words, $M_r$ is the complement of $\Phi^{-1}( \mathbb{R}^3\setminus \bar{B}_r^{\mathbb{E}}(0))$ in $M$. For $r>s>1$, we define $\mathcal{A}_{s,r}:= M_r\setminus \overline{M}_s$.

Given $r_1>r_0>1$, for each $i$, we can always choose a harmonic function $u^i$ asymptotic to $x^i$ such that the average of $u^i$ over $\mathcal{A}_{r_0,r_1}$ is zero, and every other possible choice for $u^i$ differs from this one by a constant. Theorem~\ref{t:mass} allows us to control the supremum of these ``normalized'' functions $u^i$ on $M_{r_1}$, solely in terms of the asymptotic flatness parameters.

\begin{proposition}\label{prop:sup}
Let $(M, g)$ be an orientable complete $(b,\tau)$-asymptotically flat manifold with $H_2(M,\mathbb{Z})=0$, $R_{g}\geq 0$, and mass $m(g) \leq\bar{m}$. For sufficiently large $r_0$ (depending on $b$ and $\tau$) and any $r_1>r_0$, there exists
 a constant $C(r_0, r_1, b,\tau,\bar{m})$ such that
\begin{equation}\label{e:supbound}
\sup_{M_{r_1}}|{u}|\leq C(r_0, r_1, b,\tau,\bar{m}),
\end{equation}
where ${u}$ is an asymptotically linear harmonic function (as defined in Section~\ref{sec2}) whose average over $\mathcal{A}_{r_0,r_1}$ is zero.
\end{proposition}
\begin{proof}
First, we only need to choose $r_0$ large enough (depending on $b$, $\tau$) so that $g$ is uniformly equivalent to the Euclidean metric on $M\setminus M_{r_0}$. Let ${u}$ be an asymptotically linear harmonic function whose average over $\mathcal{A}_{r_0,r_1}$ is zero.
By combining the straightforward estimate
\begin{equation}\label{eq:straight}
\left|\nabla\sqrt{|\nabla {u}|}\right|^2\leq\frac{|\nabla^2 {u}|^2}{4|\nabla {u}|},
\end{equation}
with the mass inequality (Theorem~\ref{t:mass}) and nonnegative scalar curvature, one finds
\begin{equation}\label{uyp0}
\int_{M} \left|\nabla\sqrt{|\nabla {u}|}\right|^2 dV \leq 4\pi m(g) \le 4\pi \bar{m}.
\end{equation}
Note that the hypotheses of orientability and vanishing second homology are needed to invoke Theorem~\ref{t:mass}.

We would like to use this to find an $L^2$ bound for $\nabla {u}$ on an annulus.  We start by using a uniform Sobolev inequality for asymptotically flat ends. Specifically, \cite[Lemma 3.1]{SchoenYau} applied to the function $\sqrt{|\nabla {u}|}-1$ gives us
\begin{equation}\label{eq:afpoin}
\int_{M \setminus M_{r_0/2}}\left(\sqrt{|\nabla {u}|}-1\right)^6 dV\leq C_1(r_0,b,\tau)\left(\int_{M \setminus M_{r_0/2}} \left|\nabla(\sqrt{|\nabla {u}|}-1)\right|^2dV\right)^3.
\end{equation}
Technically, \cite[Lemma 3.1]{SchoenYau} is only stated for compactly support functions.\footnote{Lemma 3.1 of~\cite{SchoenYau} also assumes a stronger definition of asymptotic flatness, but it still holds with our definition.}
However, $\sqrt{|\nabla {u}|}-1=O_1(|x|^{-\tau})$ lies in $L^6(M\setminus M_{r_0/2})$ and its gradient lies in $L^2(M\setminus M_{r_0/2})$. Thus, $\sqrt{|\nabla {u}|}-1$ can be suitably approximated by functions with compact support to obtain the desired estimate.

Applying the easy estimate $(a+b)^6\le 2^6 (a^6 + b^6)$ with $a= \sqrt{|\nabla {u}|}-1$ and $b=1$, this leads to a $L^3$ bound for $\nabla {u}$ on an annulus,
 \begin{equation}\label{-ojh}
\int_{\mathcal{A}_{r_0/2,2r_1}}|\nabla {u}|^3\,dV \leq 64\left( \int_{\mathcal{A}_{r_0/2, 2r_1}}\left(\sqrt{|\nabla {u}|}-1\right)^6 dV +
 \int_{\mathcal{A}_{r_0/2, 2r_1}} dV\right).
\end{equation}
Combining \eqref{uyp0},  \eqref{eq:afpoin}, \eqref{-ojh}, and the H\"{o}lder inequality then implies that
\begin{equation}\label{adfew}
\int_{\mathcal{A}_{r_0/2, 2r_1}}|\nabla {u}|^2 \,dV \leq C_2(r_0, r_1, b,\tau,\bar{m}).
\end{equation}
Since $u$ is a $g$-harmonic function and we have uniform $C^2$ control on $g$ over $\mathcal{A}_{r_0/2, 2r_1}$, we have the estimate
\begin{equation}\label{estHess}
\int_{\mathcal{A}_{r_0, r_1}}|\nabla^2 {u}|^2  \,dV \le  C_3(r_0, r_1, b, \tau) \int_{\mathcal{A}_{r_0/2, 2r_1}}|\nabla {u}|^2 \,dV.
\end{equation}
This can be proved in a standard way using a cutoff function equal to $1$ on $\mathcal{A}_{r_0, r_1}$ and supported in $\mathcal{A}_{r_0/2, 2r_1}$, and then integrating by parts. The point is that we can obtain an estimate with $|\nabla u|^2$ on the right instead of $u^2$ (as in the usual elliptic estimate) because $\Delta_g$ has no zeroth order term.

Next we  use our assumption that  $u$ has zero average over $\mathcal{A}_{r_0,r_1}$ to invoke the Poincar\'{e} inequality~\cite{GT},
\begin{equation}\label{Poincare}
\int_{\mathcal{A}_{r_0, r_1}}| {u}|^2 \le C_4(r_0, r_1, b, \tau)  \int_{\mathcal{A}_{r_0, r_1}}|\nabla{u}|^2.
\end{equation}
Putting together~\eqref{adfew}, \eqref{estHess}, and \eqref{Poincare}, we see that the Sobolev $H^2(\mathcal{A}_{r_0, r_1})$-norm of $u$ is bounded in terms of $r_0$, $r_1$, $b$, $\tau$, and $\bar{m}$. In three dimensions, $C^0\subset H^2$, so the corresponding Sobolev inequality on $\mathcal{A}_{r_0, r_1}$, whose constant depends only on $r_0$, $r_1$, $b$, and $\tau$, gives us a bound on $ \sup_{\mathcal{A}_{r_0, r_1}} |u|$, which then gives us the desired  bound on $\sup_{M_{r_1}} |u|$ by the maximum principle.
\end{proof}

The semi-global supremum control for the asymptotically linear harmonic functions may be parlayed into a global gradient bound, and with the aid of the mass inequality this further yields a global $L^2$ Hessian estimate. As an initial step towards the gradient bound, we show that the gradients of $u^1, u^2, u^3$  are uniformly well-approximated by the coordinate directions as we approach infinity.

\begin{lemma}\label{prop:Bartnik}
Let $(M, g)$ be an orientable complete $(b,\tau)$-asymptotically flat manifold with $H_2(M,\mathbb{Z})=0$, $R_{g}\geq 0$, and mass $m(g) \leq\bar{m}$.
For sufficiently large $r_0$ (depending on $b$ and $\tau$), there exists a constant $C(r_0, b,\tau,\bar{m})$ such that
\begin{equation}\label{567ty}
|\nabla u^i -\partial_{x^i}|(x)\leq C(r_0, b,\tau,\bar{m}) |x|^{-\tau}\text{ on } M\setminus M_{r_0},
\end{equation}
for $i=1,2,3$, where $u^i$ is a $g$-harmonic function on $M$ asymptotic to the coordinate $x^i$.
\end{lemma}

\begin{proof}
Fix an $i$, let $r_0>1$ be large enough so that Proposition~\ref{prop:sup} holds, and choose $r_1=2r_0$. Later we will see how much larger $r_0$ needs to be. Since we are proving an approximation for the gradient of $u^i$, we may assume without loss of generality that the average of $u^i$ over $\mathcal{A}_{r_0,r_1}$ is zero, since we can subtract a constant as needed.

Observe that asymptotic flatness implies that
\begin{equation}
 |\nabla x^i -\partial_{x^i}|(x)\leq C_1(r_0, b, \tau)  |x|^{-\tau}\text{ on }M\setminus M_{r_0}.
\end{equation}
So it suffices to show that we can appropriately bound the gradient of $v:=u^i-x^i$ on  $M\setminus M_{r_0}$. As described in~\eqref{hasym}, we know that $v\in C^{2,\alpha}_{1-\tau}(M\setminus M_{r_0/2})$ for any $\alpha\in(0,1)$, and thus it satisfies the following estimate, which can be thought of as an improved version of a weighted elliptic Schauder estimate on an exterior coordinate chart:
\begin{equation}\label{weighted_estimate}
\| v\| _{C^{2,\alpha}_{1-\tau}(M\setminus M_{r_0})}  \le C_2(r_0, b, \tau) \left(\| \Delta_g v\|_{C^{0,\alpha}_{-1-\tau}(M\setminus M_{r_0/2})} + \| v\|_{L^1(\mathcal{A}_{r_0/2, r_0})}\right),
\end{equation}
for sufficiently large $r_0$.
This estimate can be seen from the proof of \cite[Lemma A.41]{Lee}, for example.
Since $\Delta_g v = \Delta_g (u^i-x^i) = -\Delta_g x^i$, we see that the first term on the right side of~\eqref{weighted_estimate} is bounded by some constant $C_3(r_0, b, \tau)$. Finally, the second term on the right side of~\eqref{weighted_estimate} is bounded by some constant $C_4(r_0, b, \tau, \bar{m})$ thanks to the supremum bound from Proposition~\ref{prop:sup} and the maximum principle. The result then follows from the definition of the weighted H\"{o}lder norm.
\end{proof}

The following Hessian estimate is the technical core of our arguments and will be used frequently. It arises from the mass inequality \eqref{t:massest}, and a global gradient bound that follows from the previous proposition and the Cheng-Yau gradient estimate.

\begin{proposition}[Global gradient bound and $L^2$ Hessian bound]
\label{lem:L2est}
Let $(M, g)$ be an orientable complete $(b,\tau)$-asymptotically flat manifold with $H_2(M,\mathbb{Z})=0$, $R_{g}\geq 0$, and mass $m(g) \leq\bar{m}$.
 Moreover, assume that $\mathrm{Ric}(g)\geq -2\kappa g$, where $\kappa>0$.
There exists a constant
  $C_{\mathrm{g}}(b,\tau,\bar{m},\kappa)$ such that
\begin{equation}
\label{globalgradient}
\sup_{M}|\nabla u|\leq C_{\mathrm{g}}(b,\tau,\bar{m},\kappa),
\end{equation}
and
\begin{equation}
\label{eq:L2est}
\int_M|\nabla ^2 u|^2\,dV \leq 16\pi C_{\mathrm{g}}(b,\tau,\bar{m},\kappa)\cdot m(g),
\end{equation}
where $u$ is an asymptotically linear harmonic function on $(M, g)$.
\end{proposition}
\begin{proof}
Select a particular $r_0$, depending on $(b,\tau)$, large enough so that Lemma~\ref{prop:Bartnik} holds. As in the proof of Lemma~\ref{prop:Bartnik}, we then choose $r_1=2r_0$ and assume without loss of generality that the average of $u^i$ over $\mathcal{A}_{r_0,r_1}$ is zero.  By Lemma \ref{prop:Bartnik}, it follows that
\begin{equation}\label{1382hgn}
\sup_{M\setminus M_{r_0}}|\nabla u|\leq C_1(b, \tau,\bar{m}).
\end{equation}
In order to obtain the gradient bound on $M_{r_0}$, we may apply the Cheng-Yau gradient estimate (Theorem~\ref{t:CY}) to see that for any $q\in M$ and radius $s>0$, we have
\begin{equation}\label{///}
\sup_{B_{s}(q)}|\nabla u|\leq C_2(s,\kappa) \sup_{B_{2s}(q)}|u|.
\end{equation}
Choose $s=s(b, \tau)$ small enough so that $B_{2s}(q)\subset M_{r_1}$ for any $q\in M_{r_0}$. Then since $u$ has average zero over $\mathcal{A}_{r_0,r_1}$, Proposition \ref{prop:sup} provides a uniform bound on $\sup_{B_{2s}(q)}|u|$, yielding
\begin{equation}
\sup_{M_{r_0}}|\nabla u|\leq C_3(b,\tau,\bar{m},\kappa).
\end{equation}
Together with \eqref{1382hgn}, we achieve the global gradient bound~\eqref{globalgradient}.

To prove~\eqref{eq:L2est}, the topological condition allows us to apply the mass inequality (Theorem~\ref{t:mass}), which combines with the assumption of nonnegative scalar curvature to produce
\begin{equation}\label{jjjjjj1}
\int_M\frac{|\nabla^2 u|^2}{|\nabla u|}\,dV\leq 16\pi m(g).
\end{equation}
Putting this together with~\eqref{globalgradient}, the result~\eqref{eq:L2est} follows.

\end{proof}

\section{Estimates Along Geodesics}
\label{sec4} \setcounter{equation}{0}
\setcounter{section}{4}

The goal of the present section is to transform the $L^2$ Hessian bound~\eqref{eq:L2est} into more usable estimates along certain minimizing geodesics. By taking advantage of geodesics connecting points within the inner region of $M$ to points lying in the asymptotic end, where uniform estimates hold, we can integrate along those geodesics to obtain useful bounds for the asymptotically linear harmonic functions in the inner region of $M$.
These bounds lead to the quantitative almost Pythagorean theorem in the next section, which serves as the basis for Gromov-Hausdorff convergence to Euclidean space.


We make the following convenient definition describing the family containing the sequence of objects described in the statement of our main theorem, Theorem~\ref{t:main}.

\begin{definition}\label{def:class}
Fix parameters $b,\bar{m},\kappa>0$, $\tau\in(\tfrac{1}{2},1)$, and a point $\mathbf{p}\in\mathbb{R}^3\setminus \bar{B}_1^{\mathbb{E}}(0)$. Let $\mathcal{M}(b,\tau,\bar{m},\kappa,\mathbf{p})$ denote the
 set of triples of pointed orientable complete $(b,\tau)$-asymptotically flat $3$-manifolds $(M,g,p)$ such that
$R_g \geq 0$, $\mathrm{Ric}(g)\geq -2\kappa g$, $H_2(M,\mathbb{Z})=0$, $m(g)\leq\bar{m}$, and $\mathbf{p} = \Phi(p)$, where  $\Phi$ is the asymptotic coordinate chart of $(M,g)$ as in Definition~\ref{def:AF}. We will often use $\mathcal{M}$ as an abbreviation for $\mathcal{M}(b,\tau,\bar{m},\kappa,\mathbf{p})$  when the meaning is clear from context.

For each $(M, g, p)\in\mathcal{M}$, we define $u^1, u^2, u^3$ to be the unique $g$-harmonic functions asymptotic to the coordinates $x^1, x^2, x^3$ defined by $\Phi$, normalized so that $(u^1(p), u^2(p), u^3(p))= (0,0,0)$.
\end{definition}

\textbf{Convention:} since we will always be working in a fixed $\mathcal{M}=\mathcal{M}(b,\tau,\bar{m},\kappa,\mathbf{p})$ from now on, we suppress any dependence on the quantities $b$, $\tau$, $\bar{m}$, and $\kappa$ when we talk about what quantities a constant $C$ must depend on, as well as whenever we say that something is ``sufficiently large'' or ``sufficiently small".

With this convention in place, Theorem~\ref{t:main} can be rephrased as follows. Given $\mathcal{M}$ and any $r, \varepsilon>0$, there exists $\delta>0$ such that for all $(M, g, p)\in\mathcal{M}$, if $m(g)<\delta$, then the Gromov-Hausdorff distance between the geodesic ball $B_r(p)$ and the  Euclidean ball $B^{\mathbb{E}}_r(0)$ is less than $\varepsilon$.

For our proof, we will need to use points $q_{\pm i}$ that are far from $B_r(p)$ in the three coordinate directions, where we have strong estimates from asymptotic flatness, and then we will integrate inwards to obtain estimates in $B_r(p)$.

\textbf{Notation:} for $L>1$ and for $i=1,2,3$, we define points $q_{+i}:=\Phi^{-1} (Le_i)$ and  $q_{-i}:=\Phi^{-1}(-Le_i)$, where $e_1=(1,0,0)$,  $e_2=(0,1,0)$, and $e_3=(0,0,1)$. Whenever we introduce the parameter $L$, we are also introducing the points~$q_{\pm i}$.





The next lemma, which may be thought of as a refinement of Lemma~\ref{prop:Bartnik}, tells us that at any point $q_i^*$ near the far away point $q_{\pm i}$, $\nabla u^i$ can be approximated by the tangent vector of the minimizing geodesic connecting $q_i^*$ to a point in $B_r(p)$, while the other $\nabla u^j$'s are nearly orthogonal to it.

\begin{lemma}\label{hhhh}
Let $(M,g,p)\in \mathcal{M}$ and let $r, \varepsilon>0$. For sufficiently large $L$, depending on both $r$ and $\varepsilon$, the following holds. Fix $i=1,2,3$.
For any minimizing geodesic $\sigma_i$ from any $x^*\in B_{r}(p)$ to any $q_i^*\in B_{1}(q_{\pm i})$, we have
 \begin{equation}\label{eq:tangentcloseness}
|\pm \nabla u^i-\sigma_{i}'|(q_i^*)+\sum_{j\ne i}|\langle \nabla u^j,\sigma_{i}'\rangle| (q_{i}^*)<\varepsilon.
\end{equation}
\end{lemma}

\begin{proof}
Let $(M, g, p)\in \mathcal{M}$, fix $r>0$, and consider the case when $q_i^*\in B_{1}(q_{+i})$.  The case involving $q_{-i}$ is similar. Thanks to Lemma~\ref{prop:Bartnik} and the gradient bound~\eqref{globalgradient}, the current lemma really just depends on the relationship between the coordinate directions and the geodesics. To be precise,
\begin{equation}
|\nabla u^i -\sigma_{i}'|(q_i^*)\leq|\nabla u^i -\partial_{x^i}|(q_i^*)
+|\partial_{x^i}-\sigma_{i}'|(q_i^*),
\end{equation}
and then Lemma~\ref{prop:Bartnik}, together with uniform asymptotics, implies that $|\nabla u^i -\partial_{x^i}|(q_i^*)$ becomes small as $L$ becomes large. So it suffices to control $|\partial_{x^i}-\sigma_{i}'|(q_i^*)$ if we want to control the left-hand side of the inequality. Meanwhile,
\begin{align}
\begin{split}
|\langle \nabla u^j,\sigma_{i}'\rangle|(q_{i}^*)\leq& |\langle\nabla u^j,\partial_{x^i}\rangle|(q_i^*)+|\nabla u^j||\partial_{x^i}-\sigma_{i}'|(q_i^*)\\
\leq &|\langle \partial_{x^j},\partial_{x^i}\rangle|(q_i^*)
+|\partial_{x^i}||\nabla u^j -\partial_{x^j}|(q_i^*)+|\nabla u^j||\partial_{x^i}-\sigma_{i}'|(q_i^*).
\end{split}
\end{align}
As $L$ becomes large, the first term of the right-hand side becomes small if $j\ne i$, and the second term becomes small because of  Lemma~\ref{prop:Bartnik} again. Finally, the gradient bound~\eqref{globalgradient} tells us that it suffices to control $|\partial_{x^i}-\sigma_{i}'|(q_i^*)$ if we want to control the third term.

To summarize, we need only prove that $|\partial_{x^i}-\sigma_{i}'|(q_i^*)$ can be made arbitrarily small by choosing $L$ sufficiently large (depending also on $r$). This is just a property of uniform asymptotic flatness. Suppose that the desired statement is false. Then there exists an $\varepsilon_0 >0$,  a sequence of $(M_k,g_k, p_k)\in\mathcal{M}$ and $L_{k}\rightarrow\infty$, and points $q_{i,k}^*\in B_1(q_{+ i,k})$, $x_{k}^* \in B_{r}(p_k)$  and minimizing geodesics $\sigma_{i,k}$ from $x_k^*$ to $q_{i,k}^*$  such that
\begin{equation}
|\partial_{x^i}-\sigma_{i,k}'|_{g_k}(q_{i,k}^*)\geq \varepsilon_0.
\end{equation}
We will obtain a contradiction from a blow-down argument. There exists a radius $\bar{r}$ independent of $k$ that is large enough so that $\mathcal{S}_{\bar{r}}$ encloses $B_r(p_k)$.
Define $\psi_k: \mathbb{R}^3\setminus B^{\mathbb{E}}_{L_k^{-1}\bar{r}}(0) \rightarrow  \mathbb{R}^3\setminus B^{\mathbb{E}}_{\bar{r}}(0)$ to be the scaling diffeomorphism given by $\psi_k( \tilde{x}) = L_k \tilde{x}$. It follows from uniform asymptotic flatness that
\begin{equation}
\left( \mathbb{R}^3\setminus B^{\mathbb{E}}_{L_k^{-1}\bar{r}}(0) , \tilde{g}_k:=L_k^{-2} (\Phi_k^{-1}\circ \psi_k)^* g_k\right)\xrightarrow{C^{2}_{\mathrm{loc}}} \left( \mathbb{R}^3 \setminus \{0\},\delta\right),
\end{equation}
and
\begin{equation}
\varepsilon_0 \leq|\partial_{x^i}-\sigma_{i,k}'|_{g_k}(q_{i,k}^*)
=|\partial_{\tilde{x}^i}-\tilde{\sigma}_{i,k}'|_{\tilde{g}_k}(\tilde{q}_{i,k}^*),
\end{equation}
where $\tilde{q}_{i,k}^*=(\psi_k^{-1}\circ\Phi_k)(q_{i,k}^*)$ and  $\tilde{\sigma}_{i,k}$ is $\psi_k^{-1}\circ\Phi_k$ composed with the part of $\sigma_{i,k}$ lying outside $M_{\bar{r}}$. Observe that our definitions imply that $\tilde{q}_{i,k} \to e_i$.

Fix a small $\delta>0$. The assumption that $\mathcal{S}_{\bar{r}}$ encloses $B_r(p_k)$ guarantees that for large enough $L_k$,
$\tilde{\sigma}_{i,k}$ must connect  $\tilde{q}_{i,k}^*$ to the sphere $S_{\delta}$. Since $\tilde{\sigma}_{i,k}$ is a geodesic with respect to $\tilde{g}_k$, it must have a subsequence that converges to a Euclidean geodesic $\tilde{\sigma}_{i, \infty}$ that connects $e_i$ to a point on $S_{\delta}$.
Since we understand the direction of this line segment at $e_i$ well, it is easy to see that we can force $|\partial_{\tilde{x}^i}-\tilde{\sigma}_{i,\infty}'|(e_i) <\tfrac{1}{2}\varepsilon_0$ by choosing $\delta$ sufficiently small. It then follows that for large enough $k$, we have
\begin{equation}
|\partial_{\tilde{x}^i}-\bar{\sigma}_{i,k}'|_{\tilde{g}_k}(\tilde{q}_{i,k}^*)
<{\varepsilon_0},
\end{equation}
which is a contradiction.
\end{proof}

Proposition~\ref{lem:L2est} tells us that $m(g)$ bounds the $L^2$ integral of $|\nabla^2 u|$. In the following proposition, we will apply the segment inequality (Theorem~\ref{segmentlemma}) in a similar manner to how Cheeger and Colding applied it in~\cite{CC}, in order to obtain estimates that hold along certain geodesics. We will also integrate along geodesics that go out to (near) $q_{\pm i}$ in order to bring the estimates from Lemma~\ref{hhhh} inward from the asymptotic regime to certain geodesic segments inside the ball $B_r(p)$. This proposition will serve as the main technical tool to establish the almost Pythagorean identity in the next section.

\begin{proposition}\label{prop:technical}
Let $(M,g,p) \in \mathcal{M}$, $0<\rho<\min(r,1)$, and  $\varepsilon>0$. Then for sufficiently small $m(g)$, depending on $r$, $\rho$, and $\varepsilon$, the following holds. For each $i=1, 2, 3$, let $\Sigma^i(y)$ denote the $u^i$ level set containing $y$. For any points $x$ and $y$ satisfying $B_{\rho}(x),B_{\rho}(y)\subset B_{r}(p)$, there exist points $x^*\in B_\rho(x)$, $y^*\in B_\rho(y)$, $z\in \Sigma^i(y)$, and a minimizing geodesic $\sigma$ from $x^*$ to $z$ such that for any $s\in[0,d(x^*,z)]$ and any minimizing geodesic $\gamma_s$ from $y^*$ to $\sigma(s)$, we have
\begin{equation}\label{eq:prop1.3}
\int_0^{d(x^*,z)}\int_0^{d(y^*,\sigma(s))}\sum_{j=1}^3|\nabla^2 u^j|(\gamma_s(t))\,dt\,ds< \varepsilon.
\end{equation}
Moreover, for some choice of $\pm$, we have
\begin{equation}\label{eq:prop1.1}
\int_0^{d(x^*,z)}\left(|\pm \nabla u^i-\sigma'|+ \sum_{j\ne i} |\langle\nabla u^j,\sigma'\rangle|\right) (\sigma(s))\,ds< \varepsilon.
\end{equation}
\end{proposition}

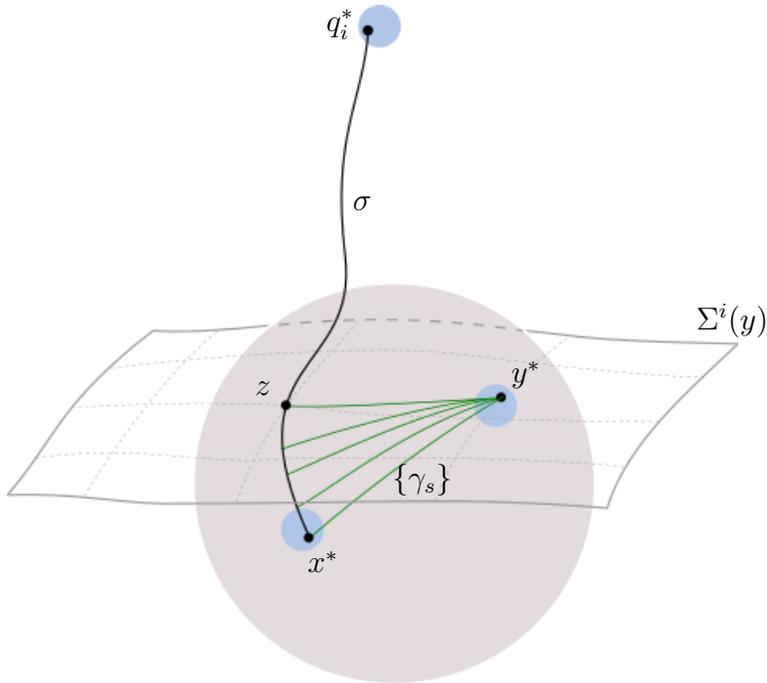
\begin{figure}
\includegraphics[scale=.6]{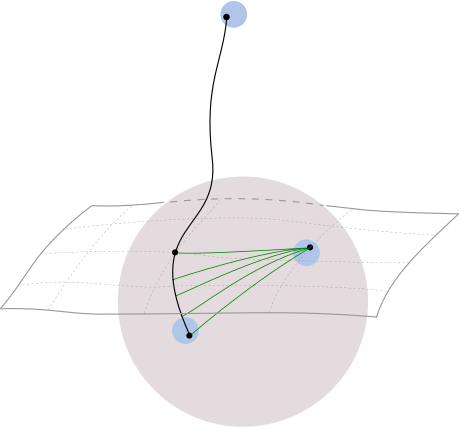}
\begin{picture}(-10,0)
\put(-160,248){\large{$q_i^*$}}
\put(-150,180){\large{$\sigma$}}
\put(-187,110){\large{$z$}}
\put(-90,115){\large{$y^*$}}
\put(-167,43){\large{$x^*$}}
\put(-135,75){\large{$\{\gamma_s\}$}}
\put(-20,135){\large{$\Sigma^i(y)$}}
\end{picture}
\caption{A depiction of the objects appearing in the proof of Proposition \ref{prop:technical}. Here $\Sigma^i(y)$ is the level set of $u^i$ passing through $y$, and the large purple ball represents $B_r(p)$, whereas the small blue balls represent $B_\rho(x)$, $B_\rho(y)$, and $B_\rho(q_{+i})$.}\label{pic:technical}
\end{figure}

\begin{proof}
Let $(M,g,p) \in \mathcal{M}$, $0<\rho<\min(r,1)$, and  $\varepsilon>0$, and fix $i$. Suppose that $x,y$ are points such that $B_{\rho}(x),B_{\rho}(y)\subset B_{r}(p)$. Our goal is to show that for any $\varepsilon>0$, we can find points $x^*$, $y^*$, $z$ and geodesics such that the estimates~\eqref{eq:prop1.3} and~\eqref{eq:prop1.1} hold for sufficiently small $m(g)$.

In order to find these points and geodesics, we fix $L>r$, which gives rise to points $q_{\pm i}$, and later we will see how large $L$ needs to be (depending on $r$, $\rho$, and $\varepsilon$) for the following constructions to give the desired estimates. For ease of notation, let us define $f:=\sum_{j=1}^3|\nabla^2 u^j|$.  As long as $L>r$,  we can apply the Cheeger-Colding segment inequality (Theorem~\ref{segmentlemma}) to the function $f$ on $B_{8L}(p)$ with $\Omega_1 = B_\rho(y)$ and $\Omega_2 =B_{4L}(p)$ to obtain
 \begin{equation}
 \int_{B_{\rho}(y)\times B_{4L}(p)}\mathcal{F}_{f} (\xi_1, \xi_2)\, d\xi_1\, d\xi_2
\leq C_{\mathrm{s}}(4L)\left(|B_{\rho}(y)|+|B_{4L}(p)|\right)\int_{B_{8L}(p)} f\,dV,
 \end{equation}
where $d\xi_1$, $d\xi_2$, and $dV$ are all notations for the Riemannian volume measure with respect to $g$.  Applying the mean value-type inequality \eqref{MV} to the $d\xi_1$ integral on the left, we see that there exists  $y^*\in B_{\rho}(y)$ such that
\begin{equation}\label{eq:prop2-1}
\int_{B_{4L}(p)}\mathcal{F}_{f}(y^*,\xi_2)\,d\xi_2 \leq
2 C_{\mathrm{s}}(4L)\left(1+\frac{|B_{4L}(p)|}{|B_{\rho}(y)|}\right)\int_{B_{8L}(p)}f\, dV.
\end{equation}
We apply the segment inequality again, but this time to the function $f_{y^*}(\xi):=f(\xi)+ \mathcal{F}_{f}(y^*,\xi)$ on $B_{4L}(p)$ with
$\Omega_1 = B_\rho(x)$ and $\Omega_2=B_\rho(q_{+i})$. Note that for sufficiently large $L$, we have $B_\rho(q_{+ i})\subset B_{2L}(p)$, so the segment inequality applies. Thus
\begin{align}
\begin{split}
\int_{B_{\rho}(x)\times B_{\rho}(q_{+ i})}&\left(\mathcal{F}_{f} (\xi_1,\xi_2)+ \mathcal{F}_{\mathcal{F}_{f}(y^*,\cdot)}(\xi_1, \xi_2)\right)\, d\xi_1\, d\xi_2 \\
\le &  C_{\mathrm{s}}(2L)\left(|B_{\rho}(x)|+|B_{\rho}(q_{+ i})|\right)
\int_{B_{4L}(p)}  \left(f+ \mathcal{F}_{f}(y^*,\cdot) \right)\,dV.
\end{split}
\end{align}
This time we apply the mean value-type inequality \eqref{MV} to the full product integral on the left to find  $x^*\in B_{\rho}(x)$ and $q^*_{i}\in B_{\rho}(q_{+ i})$ so that
\begin{align}\label{eq:prop2-2}
\begin{split}
\mathcal{F}_{f}(x^*,q^*_{ i})+&\mathcal{F}_{\mathcal{F}_{f}(y^*,\cdot)}(x^*, q^*_{ i})\\
\leq & 2  C_{\mathrm{s}}(2L)\left(\frac{1}{|B_{\rho}(x)|}+\frac{1}{|B_{\rho}(q_{+ i})|}\right)
\left(\int_{B_{4L}(p)}f\, dV+\int_{B_{4L}(p)}\mathcal{F}_{f}(y^*,\cdot) \,dV\right).
\end{split}
\end{align}
Combining \eqref{eq:prop2-1} and \eqref{eq:prop2-2} then produces
\begin{align}
\begin{split}
\mathcal{F}_{f}(x^*,q^*_{ i}) &+\mathcal{F}_{\mathcal{F}_{f}(y^*,\cdot)}(x^*, q^*_{ i})\\
&\le C_1(L) \left(\frac{1}{|B_{\rho}(x)|}+\frac{1}{|B_{\rho}(q_{+ i})|}\right)\left(1+\frac{|B_{4L}(p)|}{|B_{\rho}(y)|}\right)
\int_{B_{8L}(p)}f \, dV\\
&\le C_1(L) \left(\frac{1}{|B_{\rho}(x)|}+\frac{1}{|B_{\rho}(q_{+ i})|}\right)\left(1+\frac{|B_{4L}(p)|}{|B_{\rho}(y)|}\right) |B_{8L}(p)|^{\frac{1}{2}} \left(\int_{B_{8L}(p)} f^2\,dV\right)^{\frac{1}{2}} \\
&\le C_1(L) \left(\frac{1}{|B_{\rho}(x)|}+\frac{1}{|B_{\rho}(q_{+ i})|}\right)\left(1+\frac{|B_{4L}(p)|}{|B_{\rho}(y)|}\right) |B_{8L}(p)|^{\frac{1}{2}} \left(144\pi C_{\mathrm{g}} m(g)\right)^{\frac{1}{2}}, \label{eq:prop2-6}
\end{split}
\end{align}
where the second inequality follows from H\"{o}lder's inequality, and the last inequality follows from~\eqref{eq:L2est} in Proposition~\ref{lem:L2est} and the definition of $f$.

Ricci volume comparison tells us that $|B_{8L}(p)|\leq |B_{8L}^{-\kappa}|$, which is bounded above in terms of $L$. To bound the $\rho$-balls from below, note that uniform asymptotics tell us that if $B$ is any unit ball sufficiently far out in the asymptotically flat end, then $|B|> 1$ (since $\frac{4\pi}{3}$ is greater than $1$). For $\xi$ equal to any of $x$, $y$, or $p$, if $L$ is large enough, then $B_{L} (\xi)$ will contain such a ball $B$, and thus $|B_{L} (\xi)|>1$. Then by Bishop-Gromov relative volume comparison (Theorem~\ref{thm:BG}), we have
\begin{equation}
|B_\rho(\xi)| \ge \frac{|B_\rho^{-\kappa}|}{|B_L^{-\kappa}|}\cdot |B_L(\xi)| >  \frac{|B_\rho^{-\kappa}|}{|B_L^{-\kappa}|},
\end{equation}
and thus $|B_\rho(\xi)|$ is bounded below by a positive constant depending on $L$ and $\rho$. Inserting these bounds into~\eqref{eq:prop2-6}, we have
\begin{equation}\label{eq:prop2-3}
\mathcal{F}_{f}(x^*,q^*_{ i})+\mathcal{F}_{\mathcal{F}_{f}(y^*,\cdot)}(x^*, q^*_{ i})
\le  C_{2}(L, \rho)\cdot m(g)^{\frac{1}{2}}.
\end{equation}
Both summands on the left-hand side of \eqref{eq:prop2-3} are nonnegative,  so the right-hand side of \eqref{eq:prop2-3} bounds each individually. If we unwind the meaning of the bound on the second term, it tells us that for any minimizing geodesic $\sigma$ from $x^*$ to $q^*_{ i}$, we have
\begin{equation}
\int_0^{d(x^*, q^*_{ i})} \mathcal{F}_{f}(y^*,\sigma(s))\,ds \le C_{2}(L, \rho)\cdot m(g)^{\frac{1}{2}},
\end{equation}
and then unwinding the definitions further, we see that for any minimizing geodesics $\gamma_s$ from $y^*$ to $\sigma(s)$, we have
\begin{equation}
\int_0^{d(x^*,q^*_i)}\int_0^{d(y^*,\sigma(s))}  f(\gamma_s(t))\,dt\,ds \le C_{2}(L, \rho)\cdot m(g)^{\frac{1}{2}}.
\end{equation}

Suppose for now that $u^i(x^*) \le u^i(y)$. We will discuss the general case at the end of the proof. Lemma~\ref{prop:Bartnik} guarantees that for sufficiently large~$L$, $u^i(q^*_i) > u^i(y)$, and thus the geodesic $\sigma$ must pass through the level set $\Sigma^i(y)$. We choose  $z$ to be such a point of intersection, and then~\eqref{eq:prop1.3} will hold for sufficiently small $m(g)$, depending on $r$, $\rho$, $\varepsilon$, and $L$ (while $L>r$ itself  needs to be chosen sufficiently large).

Meanwhile, the bound on the first term of~\eqref{eq:prop2-3} says that
\begin{equation}\label{eq:prop2-7}
\int_{0}^{d(x^*,q^*_{i})} \sum_{j=1}^3 |\nabla^2 u^j|(\sigma(s))\,ds \le C_{2}(L, \rho)\cdot m(g)^{\frac{1}{2}},
\end{equation}
and we will use this to help us prove the remaining estimate~\eqref{eq:prop1.1}.
This is the only part of the proof that requires us to take $L$ sufficiently large, depending on $\varepsilon$, since we want access to the estimate from Lemma~\ref{hhhh}. Fortunately, we can still allow the smallness of $m(g)$ to depend on $L$ in order to compensate for this.
For each $s\in[0,{d(x^*,z)}]$, we apply the fundamental theorem of calculus and the Cauchy-Schwarz inequality, together with~\eqref{eq:prop2-7} to find
\begin{align}\label{eq:prop2-4}
|\nabla u^i - \sigma'|(\sigma(s))-|\nabla u^i - \sigma'|(q^*_{i})&\leq \int_{d(x^*,\sigma(s))}^{d(x^*,q^*_{i})}|\nabla^2 u^i|(\sigma(\tilde{s}))\,d\tilde{s} \le C_{2}(L, \rho)\cdot m(g)^{\frac{1}{2}}, \\
\label{38vnsop}
|\langle\nabla u^j,\sigma'\rangle|(\sigma(s))-|\langle \nabla u^j,\sigma'\rangle|(q^*_{i})&\leq\int_{d(x^*,\sigma(s))}^{d(x^*,q^*_{i})}|\nabla^2 u^j|(\sigma(\tilde{s}))\,d\tilde{s}
\le C_{2}(L, \rho)\cdot m(g)^{\frac{1}{2}},
\end{align}
for $j\neq i$. Integrating \eqref{eq:prop2-4} and \eqref{38vnsop} over $s$ from  $0$ to $d(x^*,z)$, we find
\begin{align}\label{eq:prop2-5}
\begin{split}
&\int_0^{d(x^*,z)}\left(|\nabla u^i - \sigma'|(\sigma(s))+
\sum_{j\ne i}|\langle\nabla u^j,\sigma'\rangle|(\sigma(s))\right)ds\\
&\phantom{-------}\leq
d(x^*,z)\left( 3C_{2}(L,\rho)\cdot m(g)^{\frac{1}{2}}
+|\nabla u^i - \sigma'|(q^*_{i})+\sum_{j\ne i}|\langle\nabla u^j,\sigma'\rangle|(q^*_{i})\right).
\end{split}
\end{align}
We claim that for some $m_0(r,L,\rho)$, if $m(g)< m_0$, then $d(p, z)< \mathcal{C}_0(r)$ for some constant $\mathcal{C}_0(r)$ that depends only on $r$. Since the proof of this claim is fairly involved, we will prove it as a separate lemma (Lemma \ref{lem:zdist}) after the end of this proof.  Since $d(p, z)< \mathcal{C}_0(r)$, we have $d(x^*, z)<\mathcal{C}_0(r)+r$ by the triangle inequality.
By Lemma \ref{hhhh}, for sufficiently large $L>r$ depending on $\varepsilon$, we have
\begin{equation}\label{eq:prop2-8}
|\nabla u^i - \sigma'|(q^*_{i})+\sum_{j\ne i}|\langle\nabla u^j,\sigma'\rangle|(q^*_{i})< \frac{\varepsilon}{2(\mathcal{C}_0(r)+r)}.
\end{equation}
We now select a specific $L$, depending on $r$ and $\varepsilon$, large enough so that the above holds. Given this choice of $L$, we just need $m(g)$ small enough so that $C_{2}(L, \rho)\cdot m(g)^{\frac{1}{2}} <\varepsilon$ in order to ensure estimate~\eqref{eq:prop1.3}, and also
$(\mathcal{C}_0(r)+r)\cdot 3C_{2}(L,\rho)\cdot m(g)^{\frac{1}{2}} <\frac{\varepsilon}{2}$ and $m(g)<m_0(r, L, \rho)$, which combines with~\eqref{eq:prop2-5} and~\eqref{eq:prop2-8} to ensure estimate~\eqref{eq:prop1.1} with the ``$+$'' sign.

Note that we can do the exact same construction using $B_\rho(q_{-i})$ in place of $B_\rho(q_{+i})$. In this case, $u^i(q^*_i)<u^i(y)$ for large $L$, and \emph{if} $u(x^*) \ge u^i(y)$, then we obtain the point $z$ and the same argument proves the desired result, except that
now~\eqref{eq:prop1.1} has the ``$-$'' sign, because of the sign in Lemma~\ref{hhhh}. However, it is possible that the $x^*$ obtained from the $q_{+i}$ construction has $u(x^*) > u^i(y)$, while the $x^*$ from the $q_{-i}$ construction has $u(x^*) < u^i(y)$, in which case neither construction yields the desired $z$. This is not really a problem since  this can only happen if $B_\rho(x)$ intersects $\Sigma^i(y)$, which is actually a good case for our applications. Nonetheless, it is always possible to find the desired $z$ by performing the two constructions simultaneously to obtain a \emph{single} $x^*$ that works in both the $q_{+i}$ construction and the $q_{-i}$ construction, and then of course we must either have $u(x^*) \le u^i(y)$ or $u(x^*) \ge u^i(y)$.
Since we make a similar argument in the proof of Proposition~\ref{prop:tripleseg} below, we omit the details here. The proof of Proposition~\ref{prop:tripleseg} should make it clear that the claim being made here is valid, and that the details can be easily filled in.
\end{proof}

The next result confirms the claim used at the end of the proof of Proposition \ref{prop:technical}, concerning the boundedness of $d(p,z)$.

\begin{lemma}\label{lem:zdist}
Following the construction and notation used in the proof of Proposition \ref{prop:technical}, for any $r>0$, there exists a constant $\mathcal{C}_0(r)$ such that for any choices of $L$ and $\rho\in (0,\min(r,1))$, there exists $m_0(r, L, \rho)$ such that if $m(g)<m_0$, then for any $x,y$ with $B_{\rho}(x),B_{\rho}(y)\subset B_{r}(p)$, the resulting $z$ obtained in the proof of Proposition \ref{prop:technical} satisfies
\begin{equation}
d(p,z)\leq\mathcal{C}_0(r).
\end{equation}
\end{lemma}

\begin{proof}
Suppose the statement is false. Then there exists $r>0$ such that for any positive integer $k$, there exist $L_k$ and $\rho_k\in(0,r)$, and a sequence
$\{((M_k)_l, (g_k)_l, (p_k)_l)\}_{l=1}^\infty$ in $\mathcal{M}$ with the property that $\lim_{l\to\infty} m((g_k)_l) =0$, and
there exist $(x_k)_l, (y_k)_l\in  (M_k)_l$ such that resulting  $(z_k)_l$ satisfies $d((p_k)_l ,(z_k)_l) >k$. In order to avoid use of double subscripts, we will ``pretend" that for each $k$, we have already selected an appropriate $l$ and just consider a single sequence $(M_k, g_k, p_k)$ with points $x_k$, $y_k$, $z_k$, and so on, but remember that we have the freedom to force $m(g_k)$ to approach zero as fast as we like (relative to $L_k$ and $\rho_k$) by choosing a suitably large $l$ for each $k$.

Without loss of generality, assume we are in the case where $z$ lies on a geodesic $\sigma_k$ from $x^*_k$ to $(q_{+i})_k$. We will obtain a contradiction by proving the following two convergence statements (for some subsequence):
\begin{align}\label{jfjfjfj}
\langle\nabla u_k^i ,\sigma'_k \rangle_{g_k}(z_k)&\rightarrow 0, \\
\label{ooiu}
|\nabla u_k^i -\sigma'_k|_{g_k}(z_k) &\rightarrow 0,
\end{align}
which are in obvious conflict.


Since $d(p_k, z_k)\to\infty$, it follows that the coordinate radius of $z_k$, which we will denote by $\lambda_k:=|\Phi_k(z_k)|$, also approaches infinity.
Then by Lemma~\ref{prop:Bartnik}, we can see that in order to prove~\eqref{jfjfjfj}, it suffices to show that
\begin{equation}\label{jfjfjfj2}
\langle \partial_{x^i} ,\sigma'_k \rangle_{g_k}(z_k)\rightarrow 0.
\end{equation}
Next we follow the same blow-down argument as in the proof of Lemma~\ref{hhhh}. There is a radius $\bar{r}$ independent of $k$ such that $\mathcal{S}_{\bar{r}}$ encloses
$B_r(p_k)$. We consider the scaling diffeomorphism $\psi_k(\tilde{x})=\lambda_k \tilde{x}$ and the metric $\tilde{g}_k:=\lambda_k^{-2}(\Phi_k^{-1}\circ\psi_k)^* g_k$ on the $\rr^3\setminus B^{\mathbb{E}}_{\lambda_k^{-1} \bar{r}}(0)$, which converges in $C^2$ to the Euclidean metric on  $\rr^3\setminus B^{\mathbb{E}}_{\delta}(0)$ for any fixed $\delta>0$. Meanwhile $\tilde{z}_k:=(\psi_k^{-1}\circ \Phi_k) (z_k)$ subconverges to some $z_\infty$ on the unit sphere, and the geodesics $\tilde{\sigma}_k:= \psi_k^{-1}\circ \Phi_k\circ \sigma_k$ subconverge to a Euclidean line segment $\sigma_\infty$ joining $z_\infty$ to a point on $S_{\delta}$.

We claim that the $i$-th coordinate of $z_\infty$ must be zero.
From the proof of Lemma~\ref{prop:Bartnik}, specifically~\eqref{weighted_estimate}, we know that there exist constants $c_k$ such that for all $\xi\in M_k\setminus (M_k)_{\bar{r}}$,
\begin{equation}\label{giabngk}
|u_k^i (\xi)- \Phi_k^i(\xi)  -c_k|\le C_1 |\Phi_k(\xi)|^{1-\tau},
\end{equation}
where $C_1$ is independent of $k$. Using the fact that $u^i_k(z_k) = u^i_k(y_k)$, we then have
\begin{equation}\label{phikzk} 
|\Phi_k^i(z_k)| \le |u_k^i(y_k)-c_k|+C_1 \lambda_k^{1-\tau}.
\end{equation}
Since $y_k\in B_r(p)$, the first term on the right is uniformly bounded in $k$ by the maximum principle. Indeed, the function $u_k^i -c_k$ is harmonic on $(M_k)_{2\bar{r}}$ and its Dirichlet boundary values on this domain are uniformly controlled by \eqref{giabngk}. Dividing~\eqref{phikzk} by $\lambda_k$, it follows that the $i$-th coordinate of $\tilde{z}_k$ converges to zero, proving the claim.

Using the claim, we see that $\sigma_\infty$ is a line segment from the point $z_\infty$ in the $x^i=0$ equator of the unit sphere to a point on $S_{\delta}$, and so this Euclidean line segment must be nearly orthogonal to $\partial_{x^i}$ at $z_\infty$. Since $\delta$ can be chosen to be arbitrarily small, it follows that
\begin{equation}
\langle \partial_{x^i} ,\sigma'_k \rangle_{g_k}(z_k) = \langle \partial_{x^i} ,\tilde{\sigma}'_k \rangle_{\tilde{g}_k}(\tilde{z}_k)\rightarrow 0,
\end{equation}
completing the proof of~\eqref{jfjfjfj2} and hence~\eqref{jfjfjfj}.

Next we prove~\eqref{ooiu}, and this is where we will use the fact that $m(g_k)\to0$ as quickly as we want. We also need $L_k\to \infty$, which is true since
\begin{equation}
k< d(p_k, z_k) \le d(p_k, x_k^*)  + d(x_k^*, z)   \le r+  d(x^*, (q_i^*)_k ),
\end{equation}
and if $L_k$ were bounded in $k$, then the right side of the inequality above would be bounded in $k$.
Since $L_k\to\infty$, we can use Lemma~\ref{hhhh} to see that
\begin{equation}
|\nabla u_k^i -\sigma'_k|_{g_k}((q_{i}^*)_k)\rightarrow0.
\end{equation}
Combining this with~\eqref{eq:prop2-4}, we obtain the desired convergence~\eqref{ooiu} as long as
\begin{equation}
C_{2}(L_k, \rho_k)\cdot m(g_k)^{\frac{1}{2}} \rightarrow0,
\end{equation}
where $C_{2}(L_k, \rho_k)$ is the same constant described in the proof of Proposition~\ref{prop:technical}.
But as explained at the beginning of the proof, it is possible to find such a sequence. Hence we have the desired contradiction and the result follows.
\end{proof}

In~\eqref{eq:prop2-7}, we saw how to obtain a Hessian bound along a geodesic connecting two points, as long as we are willing to perturb the end points. In the next proposition, we generalize this idea to  obtain Hessian bounds along  two geodesics, simultaneously, which connect three (perturbed) points. We do this so that we can obtain Hessian bounds along one of those geodesics, while simultaneously having $\{\nabla u^i\}_{i=1}^3$ almost form an orthonormal frame at one of the endpoints.

\begin{proposition}\label{prop:tripleseg}
Let $(M,g,p) \in \mathcal{M}$, $0<\rho<\min(r,1)$, and  $\varepsilon>0$. Then for sufficiently small $m(g)$, depending on $r$, $\rho$, and $\varepsilon$, the following holds.
For any points $x$ and $y$ satisfying $B_{\rho}(x),B_{\rho}(y)\subset B_{r}(p)$, there exist points $x^*\in B_\rho(x)$, $y^*\in B_\rho(y)$, and a minimizing geodesics $\gamma$ from $y^*$ to $x^*$ such that
\begin{equation}\label{eq:hessgeo2}
\int_0^{d(y^*,x^*)}\sum_{j=1}^{3}|\nabla^2 u^j|(\gamma(s))\,ds < \varepsilon,
\end{equation}
and
\begin{equation}\label{o3qjgoha}
\sum_{i,j=1}^{3}|\langle\nabla u^i,\nabla u^j \rangle(y^*)-\delta^{ij}|< \varepsilon.
\end{equation}
\end{proposition}

\begin{proof}
As in the proof of Proposition \ref{prop:technical}, let $(M,g,p) \in \mathcal{M}$, $0<\rho<r$, and $\varepsilon>0$. Suppose that $x,y$ are points such that $B_{\rho}(x),B_{\rho}(y)\subset B_{r}(p)$.  Let $L>r$ and consider the corresponding point $q_{+1}$; we only need one of these points for this construction. Later, we will see how large $L$ needs to be in terms of $r$, $\rho$, and $\varepsilon$. We claim that we can find  $x^*\in B_\rho(x)$, $y^*\in B_\rho(y)$, and  $q^*\in B_\rho(q_{+1})$ and geodesics $\gamma$ from $y^*$ to $x^*$  and $\gamma_1$ from $y^*$ to $q_{+1}$ such that
\begin{equation}\label{eq:trip}
\int_0^{d(y^*,x^*)}\sum_{j=1}^{3}|\nabla^2 u^j|(\gamma(s))\,ds  + \int_0^{d(y^*,q^*)}\sum_{j=1}^{3}|\nabla^2 u^j|(\gamma_1(s))\,ds   < \varepsilon.
\end{equation}
The first term gives us~\eqref{eq:hessgeo2} while the second term will be used to give us~\eqref{o3qjgoha} via the fundamental theorem of calculus (FTC).

We already saw in the proof of Proposition \ref{prop:technical} how to construct the desired $x^*$, $y^*$, and $\gamma$, \emph{or} the desired $y^*$, $q^*$, and $\gamma_1$ in the claim above. Specifically, see~\eqref{eq:prop2-7}. The only complication is that we want the \emph{same} $y^*$ in both constructions. Note that this is the exact same issue that was mentioned at the end of  the proof of Proposition \ref{prop:technical}. This can be accomplished by integrating over the product of three $\rho$-balls centered at $x$, $y$, and $q_{+1}$. Applying the Cheeger-Colding segment inequality (Theorem~\ref{segmentlemma}) to $f:=\sum_{j=1}^3|\nabla^2 u^j|$  on $B_{2r}(p)$
with $\Omega_1= B_\rho(x)$ and $\Omega_2= B_\rho(y)$ and then trivially integrating over an extra variable in  $B_{\rho}(q_{+1})$ gives
\begin{equation}
\int_{B_{\rho}(x)\times B_{\rho}(y)\times B_{\rho}(q_{+1})}\!\!\!\!
\mathcal{F}_{f}(\xi_1, \xi_2) \,d\xi_1\,d\xi_2\,d\xi_3 \le
C_{\mathrm{s}}(r)|B_{\rho}(q_{+1})| \left(|B_{\rho}(x)|+|B_{\rho}(y)|\right)
\int_{B_{2r}(p)} f\,dV.
\end{equation}
Similarly, for sufficiently large $L$, applying the segment inequality to  $f$  on $B_{4L}(p)$ with $\Omega_1=B_\rho(y)$ and $\Omega_2=B_{\rho}(q_{+1})$
and then trivially integrating over an extra variable in $B_{\rho}(x)$ gives
\begin{equation}
\int_{B_{\rho}(x)\times B_{\rho}(y)\times B_{\rho}(q_{+1})}\!\!\!\!
\mathcal{F}_{f}(\xi_2, \xi_3) \,d\xi_1\,d\xi_2\,d\xi_3  \le
C_{\mathrm{s}}(2L)|B_{\rho}(x)|\left(|B_{\rho}(y)| +|B_{\rho}(q_{+1})|\right)
\int_{B_{4L}(p)} f\,dV.
\end{equation}
Adding these together bounds the integral of $\mathcal{F}_{f}(\xi_1, \xi_2) +\mathcal{F}_{f}(\xi_2, \xi_3)$, so that
using the mean value-like inequality on the triple product $B_{\rho}(x)\times B_{\rho}(y)\times B_{\rho}(q_{+1})$ tells us that there must exist $x^* \in B_{\rho}(x)$, $y^* \in B_{\rho}(y)$, and $q^* \in B_{\rho}(q_{+1})$ such that
\begin{align}
\begin{split}
\mathcal{F}_{f}(x^*, y^*) +\mathcal{F}_{f}(y^*, q^*)& \le  C_{\mathrm{s}}(2L) \left( \frac{1}{|B_\rho(x)|} +  \frac{2}{|B_\rho(y)|} + \frac{1}{|B_\rho(q_{+1})|}\right) \int_{B_{4L}(p)} f\,dV   \label{eq:prop2-311} \\
& \le C_1(L, \rho)\cdot m(g)^{\frac{1}{2}},
\end{split}
\end{align}
where the second inequality follows from the same arguments as in Proposition \ref{prop:technical}, which involved volume comparison to bound the volumes of the $\rho$-balls from below, the H\"{o}lder inequality, and the Hessian estimate~\eqref{eq:L2est} from Proposition~\ref{lem:L2est}. Note that the bound on the first term of~\eqref{eq:prop2-311} already gives us~\eqref{eq:hessgeo2}, but we want to prove that~\eqref{o3qjgoha} holds simultaneously.


This is where we use the bound on the second term of~\eqref{eq:prop2-311}, and we must take $L$ to be large depending on $\varepsilon$ so that we can invoke
Lemma \ref{prop:Bartnik}. Let $\gamma_1$ be a minimizing geodesic from $y^*$ to $q^*$. Then the FTC, Cauchy-Schwarz, and the gradient bound \eqref{globalgradient} from Proposition~\ref{lem:L2est} imply that
\begin{align}
\begin{split}
\sum_{i,j=1}^{3}|\langle\nabla u^i,\nabla u^j \rangle (y^*)-\delta^{ij}|
&\le  \sum_{i,j=1}^{3}|\langle\nabla u^i,\nabla u^j \rangle(q^*)-\delta^{ij}| +
2C_{\mathrm{g}} \int_0^{d(y^*,q^*)}\sum_{j=1}^{3}|\nabla^2 u^j|(\gamma_1(s))\,ds \\
&\le  \sum_{i,j=1}^{3}|\langle\nabla u^i,\nabla u^j \rangle(q^*)-\delta^{ij}| +
2C_{\mathrm{g}} \mathcal{F}_{f}(y^*, q^*),
\end{split}
\end{align}
where we used the definition of $\mathcal{F}_{f}$ in the last step. By Lemma \ref{prop:Bartnik}, we can bound the first term on the right by $\frac{\varepsilon}{2}$ by selecting a specific sufficiently large $L$ (depending on $r$ and $\varepsilon$), and \eqref{eq:prop2-311} says that the second term on the right is bounded by $C_{\mathrm{g}} C_1(L, \rho) \cdot m(g)^{\frac{1}{2}}$, which can be made smaller than $\frac{\varepsilon}{2}$ for sufficiently small $m(g)$, depending on $\varepsilon$, $L$, and $\rho$. Hence we have~\eqref{o3qjgoha}, and we also have~\eqref{eq:hessgeo2} from the bound on the first term of~\eqref{eq:prop2-311}.
\end{proof}

In addition to being able to show that the vectors $\{ \nabla u^i\}_{i=1}^3$ are almost orthonormal at a point $y^*$ near a given point $y$, we can also show that they are almost orthonormal in  an integral sense.

\begin{lemma}\label{orthogonal}
Let $(M,g,p)\in \mathcal{M}$ and fix a radius $r>0$.
Given $\varepsilon>0$, for sufficiently small $m(g)$, we have
\begin{equation}
\int_{B_r(p)}|\langle\nabla u^i ,\nabla u^j \rangle -\delta^{ij}|\,dV< \varepsilon\quad \text{ for all }\quad i,j=1,2,3.
\end{equation}
\end{lemma}

\begin{proof}
According to \cite[Theorem 1.14]{Bessonetal}, under a Ricci curvature lower bound, there is a uniform constant $C_1$ depending on $\kappa$ and $r$ such that
\begin{equation}\label{eq:orthoint}
\int_{B_r(p)}|\varphi -\bar{\varphi}|^{2}\,dV\leq C_1(r)\int_{B_{3r}(p)}|\nabla \varphi|^2 \,dV
\end{equation}
for all $\varphi\in C^{\infty}(B_{3r}(p))$, where $\bar{\varphi}$ is the average value of $\varphi$ in $B_r(p)$. Here we follow our convention of suppressing dependence on $\kappa$. Applying this inequality with $\varphi=\langle \nabla u^i,\nabla u^j\rangle-\delta^{ij}$ produces
\begin{align}\label{798fn1}
\begin{split}
\int_{B_r(p)}|\langle \nabla u^i, \nabla u^j\rangle - \delta^{ij} -\bar{\varphi}|^{2}\,dV&\leq  C_1(r)\int_{B_{3r}(p)}|\nabla \langle \nabla u^i, \nabla u^j\rangle|^2 \,dV\\
&\leq  C_1(r) C_{\mathrm{g}}^2 \int_{B_{3r}(p)}\left(|\nabla^2 u^i|^2 +|\nabla^2 u^j|^2\right)\,dV\\
&\leq  32\pi C_1(r) C_{\mathrm{g}}^3 m(g),
\end{split}
\end{align}
by Proposition~\ref{lem:L2est}. Next, observe that for any $\rho<r$, the mean value-like inequality~\eqref{MV} implies the existence of $p^*\in B_{\rho}(p)$ such that
\begin{equation}
|\varphi(p^*) -\bar{\varphi}|^{2}\leq    \frac{C_2(r)}{ |B_{\rho}(p)| } m(g).
\end{equation}
However, we will need a refined version of this inequality in which it is valid simultaneously with a bound giving smallness of $|\varphi(p^*)|$, as in Proposition~\ref{prop:tripleseg}.

To achieve this, recall that in both the proofs of Proposition~\ref{prop:technical} and Proposition~\ref{prop:tripleseg}, we saw that
\begin{equation}
\int_{B_\rho(p)\times B_\rho(q_{+1})} \mathcal{F}_{f}(\xi_1,\xi_2 )\,d\xi_1\,d\xi_2  < C_3(r, L)\cdot m(g)^{\frac{1}{2}},
\end{equation}
where $f=\sum_{i=1}^3|\nabla^2 u^i|$ again, and the $L$ parameter determines the location of $q_{+1}$. Recall that this was proved by combining
the Cheeger-Colding segment inequality with the mass inequality. Combining this with~\eqref{798fn1} and using $m(g)\le \bar{m}$ yields
\begin{equation}
\int_{B_{\rho}(p)}\left(|\varphi -\bar{\varphi}|^{2}(\xi_1) +\int_{B_{\rho}(q_{+1})}\mathcal{F}_f (\xi_1,\xi_2 )\,d\xi_2 \right)\,d\xi_1 < C_4(r, L)\cdot m(g)^{\frac{1}{2}}.
\end{equation}
Therefore, by the  mean value-like inequality~\eqref{MV}, there is a $p^*\in B_{\rho}(p)$ such that
\begin{equation} \label{eq:smallave}
|\varphi(p^*) -\bar{\varphi}|^{2}  + \int_{B_{\rho}(q_{+k})}\mathcal{F}_{f}(p^*,\xi_2) \,d\xi_2 < C_5(r,L,\rho)\cdot m(g)^{\frac{1}{2}},
\end{equation}
where the $|B_{\rho}(p)|$ term has been absorbed into the $C_5$ with the aid of the volume lower bound argument from Proposition~\ref{prop:technical}. Finally, note that the smallness of the integral on the left-hand side is what was used in the proof of Proposition~\ref{prop:tripleseg} to establish \eqref{o3qjgoha}, so for any $\varepsilon_0$, there exists sufficiently large $L$ and then sufficiently small $m(g)$ such that $|\varphi(p^*)| < \varepsilon_0$. Hence by~\eqref{eq:smallave}, we can also force
$|\bar\varphi|< 2\varepsilon_0$. Inserting this back into \eqref{798fn1},
and using H\"{o}lder's inequality, along with the volume comparison $|B_{r}(p)|\leq|B_{r}^{-\kappa}|$, gives the desired result.
\end{proof}

\section{The Almost Pythagorean Identity}
\label{sec5} \setcounter{equation}{0}
\setcounter{section}{5}

The estimates along geodesics collected in the last section will now be used to establish a quantitative version of the Pythagorean theorem, in analogy to the original almost Pythagorean identity in \cite{C}. This result, Lemma \ref{lem:pyth} below, plays a leading role in the process of almost splitting along harmonic level sets. The fact that we can carry out this almost splitting in multiple directions, leads to Gromov-Hausdorff closeness to Euclidean space when the mass is small.




\begin{lemma}[Almost Pythagorean identity] \label{lem:pyth}


Let $(M,g,p) \in \mathcal{M}$, and let $r, \varepsilon>0$. Then for sufficiently small $m(g)$, depending on $r$ and $\varepsilon$, for any $x,y\in B_r(p)$ and any $i=1,2,3$, there exists
$z\in  \Sigma^i(y)$ such that
\begin{equation}\label{qiavjq0jud}
\left |d(x,z)^2+d(y,z)^2-d(x,y)^2 \right|< \varepsilon,
\end{equation}
and also $d(p,z)\leq\mathcal{C}(r)$ for some constant $\mathcal{C}(r)$ depending only on $r$.
\end{lemma}

\begin{proof}
Let $(M, g, p)\in\mathcal{M}$, let $r,\varepsilon>0$, fix $i$, and pick $x,y\in B_r(p)$. Choose $\rho\in (0,\min(r, 1))$  and we will see later on how much smaller $\rho$  needs to be.  In particular, $B_\rho(x), B_\rho(y)\subset B_{2r}(p)$. Let $\varepsilon_0>0$. We will see later how small it needs to be, but for now, apply Proposition \ref{prop:technical} to the points $x,y$ in the ball $B_{2r}(p)$, using this $\varepsilon_0$.  So
for sufficiently small $m(g)$ depending on $r$, $\varepsilon_0$, and $\rho$,
we can find points $x^*\in B_\rho(x)$, $y^*\in B_\rho(y)$, $z\in \Sigma^i(y)$, and a minimizing geodesic $\sigma$ from $x^*$ to $z$ such that for any $s\in[0,d(x^*,z)]$ and any minimizing geodesic $\gamma_s$ from $y^*$ to $\sigma(s)$, we have
\begin{equation}\label{eq:pyth1}
\int_0^{d(x^*,z)}\int_0^{d(y^*,\sigma(s))} |\nabla^2 u^i|(\gamma_s(t))\,dt\,ds< \varepsilon_0,
\end{equation}
and
\begin{equation}\label{eq:pyth2}
\int_0^{d(x^*,z)} |\pm\nabla u^i-\sigma'| (\sigma(s))\,ds< \varepsilon_0,
\end{equation}
for some choice of $\pm$. We will not need the terms involving $j\ne i$ here. Let us assume that this estimate holds with the ``$+$'' sign, since the ``$-$'' case is similar. Moreover, $d(p ,z)< \mathcal{C}(r)$ where $\mathcal{C}(r):= \mathcal{C}_0(2r)$, where $\mathcal{C}_0(2r)$ is given by Lemma~\ref{lem:zdist}.

Define $T:=d(x^*, z)$, so that $\sigma(T)=z$, and define $l(s):= d(y^*, \sigma(s))$, so that $\gamma_s(l(s))=\sigma(s)$. Note that
\begin{equation}\label{Tbound}
T \le  d(x^*, p) + d(p,z)< 2r+ \mathcal{C}(r),
\end{equation}
and
\begin{equation}\label{lsbound}
l(s) \le d(y^*, x^*) + d(x^*, \sigma(s)) < 4r + d(x^*, z) < 6r+ \mathcal{C}(r).
\end{equation}
It is well-known that $l'(s)$ exists for almost every $s$. By the first variation of arclength, $l'(s) = \left \langle\sigma'(s), \gamma_s'(l(s))\right \rangle$.  By the fundamental theorem of calculus (FTC),
\begin{align}\label{eq6.1}
\begin{split}
\frac{1}{2}\left(d(y^*,z)^2-d(x^*,y^*)^2\right) =&\frac{1}{2}\left(l(T)^2-l(0)^2\right)\\
=& \int_0^T l(s)l'(s) \,ds\\
=&\int_0^T l(s) \left \langle\sigma'(s), \gamma_s'(l(s))\right \rangle \,ds\\
\le& \int_0^T l(s) \left \langle \nabla u^i(\sigma(s)), \gamma_s'(l(s)) \right \rangle \,ds
+\int_0^T l(s) |\nabla u^i-\sigma'| (\sigma(s))\,ds \\
<& \int_0^T \int_{0}^{l(s)} \left \langle \nabla u^i(\gamma_s(l(s))), \gamma_s'(l(s)) \right \rangle\, dt\, ds + C_1(r)\varepsilon_0,
\end{split}
\end{align}
where we used \eqref{eq:pyth2} and \eqref{lsbound} in the last line. Next observe that for any $t\in[0,l(s)]$,  the FTC says that
\begin{equation}
\left \langle \nabla u^i(\gamma_s(l(s))), \gamma_s'(l(s)) \right \rangle
=\left \langle \nabla u^i(\gamma_s(t)), \gamma_s'(t) \right \rangle
+\int_t^{l(s)}\nabla^2 u^i(\gamma_{s}'(\tau),\gamma_{s}'(\tau))\,d\tau.
\end{equation}
Integrating this, we obtain
\begin{align}
\begin{split}
\int_0^T \int_{0}^{l(s)} &
\left \langle \nabla u^i(\gamma_s(l(s))), \gamma_s'(l(s)) \right \rangle\,dt\,ds\\
&\le \int_0^T \int_{0}^{l(s)} \left \langle \nabla u^i(\gamma_s(t)), \gamma_s'(t) \right \rangle\,dt\,ds
+\int_0^T \int_{0}^{l(s)} \int_t^{l(s)}|\nabla^2 u^i| (\gamma_s(\tau)) \,d\tau\,dt\,ds \\
&< \int_0^T  \left(u^i(\gamma_s(l(s)))  - u^i(\gamma_s(0))\right)  ds  + C_2(r) \varepsilon_0\\
&=\int_0^T  \left(u^i(\sigma(s)))  - u^i(y^*)\right)  ds  + C_2(r) \varepsilon_0, \label{jhui}
\end{split}
\end{align}
where the second inequality follows from the FTC, \eqref{eq:pyth1}, and \eqref{lsbound}.
Since $z\in \Sigma^i(y)$, we may use the FTC again to see that
\begin{align}\label{eq:utodist}
\begin{split}
u^i(\sigma(s))-u^i(y)
&= u^i(\sigma(s))-u^i(z)\\
&= u^i(\sigma(s))-u^i(\sigma(T))\\
&=-\int_s^T \langle \nabla u^i (\sigma(t)), \sigma'(t)\rangle \,dt \\
&=\int_s^T\left[-1+ \langle\sigma'(t) -\nabla u^i(\sigma(t)), \sigma'(t)\rangle\right]\,dt\\
&\le s-T +  \int_s^T | \sigma' -\nabla^i u|(\sigma(t))\,dt \\
&< s-T +  \varepsilon_0,
\end{split}
\end{align}
where we used~\eqref{eq:pyth2} in the last line. We can connect this back to the integrand of \eqref{jhui} by using the gradient bound~\eqref{globalgradient} to see that $|u^i(y^*)-u^i(y)|\le C_{\mathrm{g}} \rho$. Combing this with~\eqref{eq6.1},~\eqref{jhui}, and~\eqref{eq:utodist}
yields
\begin{align}\begin{split}
\frac{1}{2}\left(d(y^*,z)^2-d(x^*,y^*)^2\right)&<\int_0^T [(s-T) + \varepsilon_0 + C_{\mathrm{g}} \rho]\,ds + (C_1(r)+C_2(r))\varepsilon_0 \\
&= -\frac{1}{2}T^2 + T(\varepsilon_0 + C_{\mathrm{g}} \rho) + (C_1(r)+C_2(r))\varepsilon_0 \\
&\le -\frac{1}{2}d(x^*,z)^2 + C_3(r) (\varepsilon_0+ \rho).
\end{split}
\end{align}

It is clear that analogous arguments may be used to obtain a lower bound of the same form, and hence
\begin{equation}\label{28vnqod}
\left|d(x^*,z)^2+d(y^*,z)^2-d(x^*,y^*)^2\right|
<C_3(r)(\varepsilon_0 +\rho).
\end{equation}
Finally, to replace $x^*$ and $y^*$ by $x$ and $y$, note that $d(x, x^*)$ and $d(y, y^*)$ are less than $\rho$ and that the distances $d(x^*,z)$, $d(y^*,z)$, and $d(x^*,y^*)$ are all bounded in terms of $r$. From this, simple use of the triangle inequality implies that
\begin{equation}
\left|d(x, z)^2+d(y ,z)^2-d(x,y)^2\right|
<C_4(r)(\varepsilon_0 +\rho),
\end{equation}
for some new constant $C_4(r)$. Lastly, we can see that if we choose both $\varepsilon_0$ and $\rho$ to be smaller than $\frac{\varepsilon}{ 2C_4(r)}$,  then the desired estimate~\eqref{qiavjq0jud} holds for sufficiently small $m(g)$, as determined by Proposition~\ref{prop:technical} with these choices of $r$, $\varepsilon_0$, and $\rho$.
\end{proof}

The following lemma could have been stated as part of the previous lemma, but we choose to break up the exposition. It helps us to use the harmonic functions $u^i$ to approximate distances.

\begin{lemma}\label{lem:distandui}
In addition to the conclusions stated in Lemma \ref{lem:pyth}, we have the following estimates:
\begin{equation}\label{eq:distandui1}
\left|d(x,z)-|u^i(x)-u^i(z)|\right|<\varepsilon,
\end{equation}
and for any second index $j\neq i$ we have
\begin{equation}\label{eq:distandui2}
|u^j(x)-u^j(z)|<\varepsilon.
\end{equation}
\end{lemma}

\begin{proof}
Continuing on from the proof of Lemma~\ref{lem:pyth}, with the same notation, we again assume we are in the case where estimate~\eqref{eq:pyth2} holds with the ``$+$'' sign, since the ``$-$'' case is similar. Going back to Proposition~\ref{prop:technical}, this corresponds to when $u^i(x^*) \le u^i(z)$. Then by the FTC,
\begin{align}\label{eq:udiff1}
\begin{split}
\left||u^i(x^*)-u^i(z)|-d(x^*,z)\right|&=\left| (u^i(z)-u^i(x^*))-d(x^*, z)\right|\\
&=\left|\int_0^{d(x^*, z)}\left[  \langle \nabla u^i, \sigma'\rangle-  \langle \sigma', \sigma'\rangle \right](\sigma(s))\,ds\right|\\
&\leq\int_0^{d(x^*, z)}\left|\nabla u^i -\sigma'\right|(\sigma(s))ds\\
&<\varepsilon_0,
\end{split}
\end{align}
where we used~\eqref{eq:pyth2} in the last line. Meanwhile, if $j\neq i$, the FTC gives us\
\begin{equation}\label{eq:udiff2}
\left|u^j(x^*)-u^j(z)\right|
=\left| \int_0^{d(x^*, z)}\langle\nabla u^j,\sigma'\rangle(\sigma(s))\,ds\right|<\varepsilon_0,
\end{equation}
where the inequality follows from the bound on the second term in \eqref{eq:prop1.1} from Proposition~\ref{prop:technical}.

Using the gradient bound~\eqref{globalgradient} to estimate $|u^j(x)-u^j(x^*)|\le C_{\mathrm{g}} \rho$ for $j=1,2,3$, and combining this with \eqref{eq:udiff1} and \eqref{eq:udiff2}, we obtain
\begin{align} \label{eq:udiff4}
\begin{split}
\left||u^i(x)-u^i(z)|-d(x,z)\right| &\le \varepsilon_0 + (C_{\mathrm{g}}+1) \rho \\
\left| u^j(x)-u^j (z)\right| &\le \varepsilon_0 + C_{\mathrm{g}} \rho,
\end{split}
\end{align}
for $j\ne i$. From this it is clear that if we select $\rho$ and $\varepsilon_0$ sufficiently small, we obtain the desired estimates (as well as the ones from Lemma~\ref{lem:pyth}), for sufficiently small $m(g)$ depending on these choices of $r, \varepsilon_0$, and $\rho$.
\end{proof}

\section{Proof of the Main Theorem}
\label{sec6} \setcounter{equation}{0}
\setcounter{section}{6}

We will prove Theorem~\ref{t:main} by showing that for any radius $\varrho>0$ the Gromov-Hausdorff distance between $B_\varrho(p)\subset M$ and the Euclidean ball $B^\mathbb{E}_\varrho(0)$ can be made arbitrarily small by choosing the mass to be sufficiently small. To this end, consider the map $\mathbf{u}:B_r(p)\to\mathbb{R}^3$ defined by $\mathbf{u}(x)=(u^1(x),u^2(x),u^3(x))$. This is the only place where we use the convention in Definition~\ref{def:class} that
 $\mathbf{u}(p)=0$. The first step of the proof of the main result
is to demonstrate that $\mathbf{u}$ is an $\varepsilon$-isometry onto its image for sufficiently small mass. The second step is to ensure that the Hausdorff distance between the image $\mathbf{u}(B_r(p))$
and $B^\mathbb{E}_r(0)$ is less than $\varepsilon$ for sufficiently small mass. Thus, $\mathbf{u}$ will give rise to the desired Gromov-Hausdorff approximation.

\begin{theorem}\label{t:mainGH}
Let $(M,g,p)\in \mathcal{M}(b,\tau,\bar{m},\kappa,\mathbf{p})$ and fix a radius $r>0$. Given $\varepsilon>0$, there is a $\delta>0$ such that if
 $m(g)<\delta$, then for all $x, y\in B_r(p)$ we have
\begin{equation}
\left| d(x,y) - |\mathbf{u}(x)-\mathbf{u}(y)| \right|  <\varepsilon.
\end{equation}
\end{theorem}

\begin{proof}
In this proof we use the notation $\Psi(m)$ to denote any quantity that can be made arbitrarily small by choosing $m(g)$ sufficiently small, while keeping all of the other quantities $(b,\tau,\bar{m},\kappa,\mathbf{p})$, as well as $r$, fixed. It suffices to prove that
\begin{equation}\label{eq:mainsplitting}
d(x,y)^2 =  |u^1(x)-u^1(y)|^2+|u^2(x)-u^2(y)|^2+|u^3(x)-u^3(y)|^2  + \Psi(m).
\end{equation}
By Lemma \ref{lem:pyth} with $i=1$, there exists
$x'\in \Sigma^1(y) \cap B_{\mathcal{C}(r)}(p)$ such that
\begin{align}\label{eq:splitting1}
\begin{split}
d(x,y)^2&= d(x,x')^2 + d(x',y)^2   +  \Psi(m ) \\
&=  |u^1(x)-u^1(y)|^2 + d(x',y)^2 +  \Psi(m),
\end{split}
\end{align}
where we used~\eqref{eq:distandui1} from Lemma~\ref{lem:distandui}, and $u^1(x')=u^1(y)$, as well as the fact that 
$|u^1(x)-u^1(y)|\le2C_{\mathrm{g}}r$ to obtain the last line. By repeating this argument with $i=2$ for the points $x', y \in B_{\mathcal{C}(r)}(p)$, there exists $x''\in \Sigma^2(y)\cap B_{\mathcal{C}(\mathcal{C}(r))}(p)$ such that
\begin{equation}
d(x',y)^2= |u^2(x')-u^2(y)|^2 + d(x'',y)^2  +  \Psi(m).
\end{equation}
Doing it one more time with $i=3$ for the points $x'', y\in B_{\mathcal{C}(\mathcal{C}(r))}(p)$, there exists $x'''\in \Sigma^3(y)\cap B_{\mathcal{C}( \mathcal{C}(\mathcal{C}(r)))}(p)$ such
such that
\begin{equation}
d(x'',y)^2= |u^3(x'')-u^3(y)|^2 + d(x''',y)^2  +  \Psi(m).
\end{equation}
Putting the last three equations together, we have
\begin{equation}
d(x,y)^2 = |u^1(x)-u^1(y)|^2+|u^2(x')-u^2(y)|^2+|u^3(x'')-u^3(y)|^2 +d(x''',y)^2+ \Psi(m).
\end{equation}
See Figure \ref{pic:surjectiveGH} for a visualization of this construction.
\begin{figure}
\includegraphics[scale=.6]{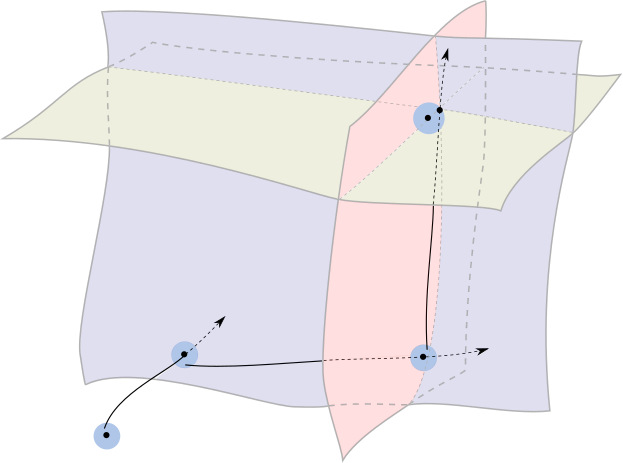}
\begin{picture}(-10,0)
\put(-133,202){\large{$y$}}
\put(-110,217){\large{$x'''$}}
\put(-138,68){\large{$x''$}}
\put(-282,70){\large{$x'$}}
\put(-329,20){\large{$x$}}
\put(-370,80){\large{$\Sigma^1(y)$}}
\put(-163,0){\large{$\Sigma^2(y)$}}
\put(-400,210){\large{$\Sigma^3(y)$}}
\end{picture}
\caption{This illustration depicts the construction used in the proof of Theorem \ref{t:mainGH}. The point $x$ is ``quasi-projected" three times onto $u^1$, $u^2$, and $u^3$ level sets passing through $y$. The projection is accomplished by shooting geodesics, denoted by $\sigma$ in the proof of Lemma \ref{lem:pyth}, towards far away points
$q_{\pm 1}$, $q_{\pm 2}$, and $q_{\pm 3}$ (not shown). Note that these geodesics start at perturbed points that lie within $B_{\rho}(x)$, $B_{\rho}(x')$, and $B_{\rho}(x'')$.}
\label{pic:surjectiveGH}
\end{figure}
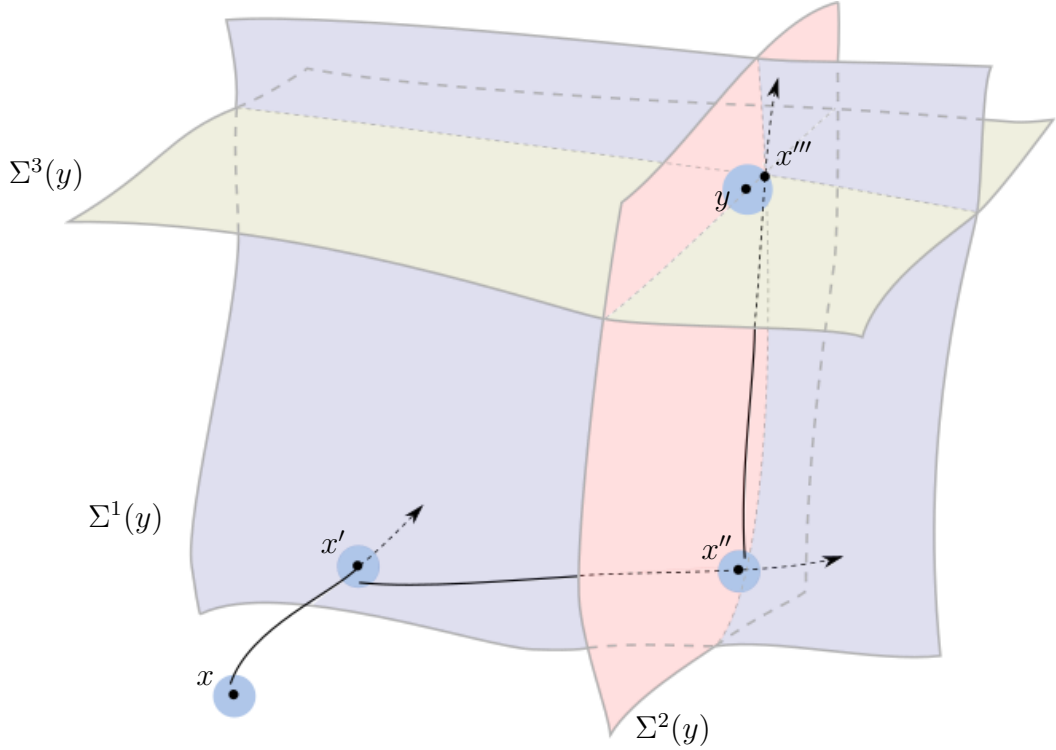
Furthermore, by repeated application of~\eqref{eq:distandui2} of Lemma \ref{lem:distandui} (as well as bounds on all of the $u^j$ differences), it follows that
\begin{equation}
d(x,y)^2 = |u^1(x)-u^1(y)|^2+|u^2(x)-u^2(y)|^2+|u^3(x)-u^3(y)|^2 +d(x''',y)^2+ \Psi(m).
\end{equation}

It remains to show that $d(x''',y)= \Psi(m)$. First observe that the triangle
inequality plus repeated application of~\eqref{eq:distandui2} of Lemma \ref{lem:distandui} implies that
\begin{align}\label{eq:udiffsmall}
\begin{split}
\sum_{i=1}^3|u^i(x''')-u^i(y)| =&|u^1(x''')-u^1(y)| + |u^2(x''')-u^2(y)| +\underbrace{|u^3(x''')-u^3(y)|}_{=0}  \\
\leq & |u^1(x''')-u^1(x'')|+|u^1(x'')-u^1(x')| + \underbrace{|u^1(x')-u^1(y)|}_{=0}\\
&+|u^2(x''')-u^2(x'')|+\underbrace{|u^2(x'')-u^2(y)|}_{=0}\\
=& \Psi(m).
\end{split}
\end{align}
This estimate is the main tool that will be used to prove that $d(x''', y)$ is small. The basic intuition is that $\{\nabla u^i\}_{i=1}^3$ is close to forming an orthonormal frame at a point near $y$, and from there we can use a Taylor expansion to control distance in terms of change in $u^i$. As usual, we must employ perturbations in order to have effective pointwise bounds.

To carry out this intuition, let
$\varepsilon_0>0$ and $\rho\in (0,\min(r,1))$, and recall that Proposition \ref{prop:tripleseg} on the ball 
$B_{\mathcal{C}(\mathcal{C}(\mathcal{C}(r)))}(p)$ states that for sufficiently small $m(g)$ depending on $\mathcal{C}(\mathcal{C}(\mathcal{C}(r)))$, $\varepsilon_0$, and $\rho$, we can find points
$x^*\in B_\rho(x''')$ and $y^*\in B_\rho(y)$, along with a minimizing geodesics $\gamma$ from $y^*$ to $x^*$ such that
\begin{equation}\label{eq:triplesegclaim}
\int_0^{d(x^*,y^*)}\sum_{i=1}^3|\nabla^2 u^i|(\gamma(s))\,ds < \varepsilon_0,
\end{equation}
and
\begin{equation}
\label{aonhhwq}
| \langle \nabla u^i, \nabla u^j\rangle(y^*) - \delta^{ij}|  < \varepsilon_0.
\end{equation}
Next, define
\begin{equation}
\bar{u}:=\sum_{i=1}^3 c_i u^i\quad\text{ where }\quad c_i:= (u^i\circ\gamma)'(0)=\langle \nabla u^i, \gamma'\rangle(y^*).
\end{equation}
We provide some motivation for $\bar{u}$. Note that \emph{if} $\{\nabla u^i\}_{i=1}^3$ were an honest orthonormal basis at $y^*$, then the $c_i$'s would be the coefficients of the unit vector $\gamma'(0)$ with respect to that basis, and thus we would have $\gamma'(0)=\nabla\bar{u}$ at $y^*$. Therefore
changes in this $\bar{u}$ along $\gamma$, which we know should be small because of~\eqref{eq:udiffsmall}, should match distances up to first order. Meanwhile, the error beyond first order can be controlled by the Hessian.

Although we do not have orthonormality, \eqref{aonhhwq} does imply that $\varepsilon_0$ can be chosen small enough to guarantee that
\begin{equation} \label{eq:coefficients}
\sum_{i=1}^3 c_i^2 > \frac{1}{2}.
\end{equation}
One can see this by applying the Gram-Schmidt process to the almost orthonormal frame $\{\nabla u^i (y^*)\}_{i=1}^3$ to produce a true orthonormal basis (close to this set of gradient vectors), and then proceed to show that the $c_i$'s are close to the coefficients of $\gamma'(0)$ in that orthonormal basis; these linear algebra computations are straightforward but tedious.
The Taylor expansion of $\bar{u}\circ \gamma$ at $0$ reads
\begin{equation}\label{eq:taylor1}
(\bar{u}\circ\gamma)(t)=(\bar{u}\circ\gamma)(0)+ (\bar{u}\circ\gamma)'(0)\cdot t+\mathcal{R}(t),
\end{equation}
where the remainder term $\mathcal{R}(t)$ satisfies the estimate
\begin{equation} \label{eq:remainder}
|\mathcal{R}(t)|\leq t\int_0^t |(\bar{u}\circ\gamma)''(s)|ds.
\end{equation}
Note that $(\bar{u}\circ\gamma)'(0)=  \sum_{i=1}^3 c_i^2$ by construction, and substituting $t=d(x^*, y^*)$  in the Taylor expansion~\eqref{eq:taylor1} then yields
\begin{equation}\label{eq:taylor2}
(\bar{u}(y^*) -\bar{u}(x^*)) +    \left(\sum_{i=1}^3 c_i^2 \right)\cdot d(x^*,y^*)+   \mathcal{R}(d(x^*, y^*)) = 0.
\end{equation}
Estimating the first term, we find that
\begin{align}\label{eq:ubardiff}
\begin{split}
|\bar{u}(y^*) -\bar{u}(x^*)| & \le \left(\sup_i |c_i|\right) \sum_{i=1}^3 |u^i (y^*)- u^j (x^*)| \\
& \le C_{\mathrm{g}} \sum_{i=1}^3\left( |u^i (y)- u^j (x''')|+2 C_{\mathrm{g}}\rho \right)\\
&= \Psi(m) + 6 C_{\mathrm{g}}^2 \rho
\end{split}
\end{align}
by~\eqref{eq:udiffsmall}. While estimating the third term of~\eqref{eq:taylor2} using~\eqref{eq:remainder} produces
\begin{align} \label{eq:rem_est}
\begin{split}
|\mathcal{R}(d(x^*, y^*))|&\le d(x^*, y^*) \left(\sup_i |c_i|\right) \int_0^{d(x^*, y^*)} \sum_{i=1}^3|\nabla^2 u^i|(\gamma(s))\,ds \\
&<
2\mathcal{C}( \mathcal{C}(\mathcal{C}(r))) C_{\mathrm{g}} \varepsilon_0,
\end{split}
\end{align}
where~\eqref{eq:triplesegclaim} was used for the second inequality. Invoking the triangle inequality and then inserting \eqref{eq:ubardiff}, \eqref{eq:rem_est}, and~\eqref{eq:coefficients} into~\eqref{eq:taylor2}, we obtain
\begin{align}
\begin{split}
d(x''', y) &\le d(x^*, y^*) + 2\rho \\
&\le 2\left[\Psi(m) + 6 C_{\mathrm{g}}^2 \rho +  2\mathcal{C}( \mathcal{C}(\mathcal{C}(r))) C_{\mathrm{g}} \varepsilon_0\right]+ 2\rho.
\end{split}
\end{align}
From this it is clear that $d(x''', y)=\Psi(m)$, because we can select $\varepsilon_0$ and $\rho$ small enough to make their contributions to the above expression as small as we like, and according to Proposition~\ref{prop:tripleseg}, we accomplish this by selecting $m(g)$ sufficiently small with respect to $r$, $\varepsilon_0$, and $\rho$.
\end{proof}

As mentioned, the previous theorem implies that for any $r, \varepsilon>0$, the Gromov-Hausdorff distance between $B_r(p)$ and  $\mathbf{u}(B_r(p))$ can be made less than $\varepsilon$ by choosing $m(g)$ sufficiently small.
But more than that, the estimates imply that the image $\mathbf{u}(B_r(p))$ lies in the Euclidean ball $B^\mathbb{E}_{r+\varepsilon}(0)$. Hence, in order to complete the proof of Theorem \ref{t:main} it only remains to show that for small enough mass, every element of $B^\mathbb{E}_r(0)$ is within a distance $\varepsilon$ from the image $\mathbf{u}(B_r(p))$.
That is, given a point $x= (x^1, x^2, x^3) \in B^\mathbb{E}_r(0)$, we must produce a point $w\in B_r(p)$ whose image under $\mathbf{u}$ is near $x$. Given the intuition that $\mathbf{u}$ is supposed to be near the map $\Phi - \mathbf{p}$, our method of producing $w$ is to start with the point $p$, follow the gradient flow for $\nabla u^1$ for time $x^1$, then follow the gradient flow for $\nabla u^2$ for time $x^2$, and finally follow the gradient flow for $\nabla u^3$ for time $x^3$. However, as has been the case throughout this paper, we will have to make small perturbations at each step in order for this to work.

For each $i=1,2,3$, let $\psi_t^i: M\rightarrow M$ denote the gradient flow for $u^i$, defined by the requirement that
$\frac{\partial}{\partial t} \psi_t^i(q) =\nabla u^i (\psi_t^i(q))$ and $\psi_0^i (q)=q$ for any $q\in M$. Note that since $u^i$ is harmonic, its gradient is divergence-free, and hence $\left(\psi_{t}^i\right)^* dV_{\psi_{t}^i(q)}=dV_q$, where $dV_q$ represents the volume form at the point $q$. Or in other words, the flow is measure-preserving. Motivated by \cite[Lemma 2.1]{KapovitchWilking}, we now show that the gradient flows may be used to approximate distances in the direction orthogonal to level sets in $M$, which is what is needed to implement the argument described above.


\begin{lemma}\label{yui3}
Let $(M,g,p) \in \mathcal{M}$, $0<\rho<\min(r,1)$, and  $\varepsilon>0$. Then for sufficiently small $m(g)$, depending on $r$, $\rho$, and $\varepsilon$, the following holds. For each $y\in M$ and $t_0>0$ with $d(p,y) + C_{\mathrm{g}} t_0 +\rho < r$, there exists $y^*\in B_\rho (y)$ for which
\begin{equation}\label{eq:Ufirst}
 \left | \mathbf{u} \left(\psi^j_{t_0} (y^*)\right)- \mathbf{u}\left(y^*\right) - t_0 e_j \right| < \varepsilon,
\end{equation}
for each $j=1,2,3$, where $e_1, e_2, e_3$ is the standard basis in $\mathbb{R}^3$.
\end{lemma}

\begin{proof}
The estimate to be proved is equivalent to
\begin{equation}\label{eq:yui3goal}
| u^i (\psi^j_{t_0} (y^*))- u^i(y^*) - \delta^{ij} t_0 | <\varepsilon,
\end{equation}
and we can estimate the left side using the FTC:
\begin{align}\label{eq:udifferences}
\begin{split}
\left| u^i (\psi^j_{t_0} (y^*))- u^i(y^*) - \delta^{ij} t_0 \right|
&= \left| \int_{0}^{t_0}\left( \frac{\partial}{\partial t} u^i(\psi_t^j(y^*)) -\delta^{ij} \right)\,dt  \right| \\
&\le
\int_{0}^{t_0}\left| \langle\nabla u^i,\nabla u^j\rangle  -\delta^{ij}\right|(\psi_t^j(y^*))\, dt.
\end{split}
\end{align}
We will use Lemma~\ref{orthogonal} to show that we can bound this quantity.


First observe that the global gradient bound~\eqref{globalgradient} of Proposition~\ref{lem:L2est} implies a uniform estimate on the distance the flow can move in a fixed time. That is, for any $\xi\in M$, we have
\begin{equation}\label{tttt}
d(\psi_{t}^{j}(\xi),\xi)\leq\int_{0}^{t}\left|\frac{\partial}{\partial s} \psi_s^j(\xi)\right|\,ds=\int_{0}^{t}|\nabla u^j(\psi_s^j(\xi))| \,ds
\leq C_{\mathrm{g}} t.
\end{equation}
Combined with the assumption $d(p,y) + C_{\mathrm{g}} t_0 +\rho < r$, it follows that if $\xi\in B_\rho(y)$, then $\psi_{t}^{j}(\xi)\in B_r(p)$ for all $t\in [0,t_0]$. Using Fubini's Theorem and the measure preserving property of the gradient flow, we have
\begin{align}
\begin{split}
\int_{B_\rho(y)}&\left (\int_{0}^{t_0}|\langle\nabla u^i,\nabla u^j\rangle -\delta^{ij}|(\psi_{t}^{j}(\xi))dt \right) dV_\xi\\
&=
\int_{0}^{t_0} \left(\int_{B_\rho(y)}|\langle\nabla u^i,\nabla u^j\rangle -\delta^{ij}|(\psi_{t}^{j}(\xi))\,\left(\psi_{t}^j\right)^* dV_{\psi_{t}^j(\xi)}\right)dt\\
&=\int_{0}^{t_0}\left( \int_{\psi_{t}^{j}\left(B_\rho(y)\right)}|\langle\nabla u^i,\nabla u^j\rangle -\delta^{ij}|(\eta)\,dV_\eta\right)  dt\\
&\leq  \int_{0}^{t_0}\left(\int_{B_{r}(p)}|\langle\nabla u^i,\nabla u^j\rangle -\delta^{ij}|\,dV\right) dt.
\end{split}
\end{align}
Combining this with~\eqref{eq:udifferences} and the mean value-like inequality~\eqref{MV}, we can find a point $y^*\in B_\rho(y)$ such that
\begin{equation}
| u^i (\psi^j_{t_0} (y^*))- u^i(y^*) - \delta^{ij} t_0 |  < \frac{2r}{|B_\rho(y)|}   \int_{B_{r}(p)}|\langle\nabla u^i,\nabla u^j\rangle -\delta^{ij}|\,dV.
\end{equation}
By Lemma \ref{orthogonal}, the right side can be made less than $\varepsilon$ for sufficiently small mass, since the $|B_\rho(y)|$ term can be handled as in the proof of Proposition~\ref{prop:technical}, giving us the desired estimate~\eqref{eq:yui3goal}.
\end{proof}


We are now ready to complete the second step in the proof of our main Theorem \ref{t:main}.

\begin{theorem}\label{surjectiveGH}
Let $(M,g,p)\in \mathcal{M}(b,\tau,\bar{m},\kappa,\mathbf{p})$.
Given $\varrho, \varepsilon>0$, there is a $\delta>0$ such that if $m(g)<\delta$, then the Euclidean ball
$B_{\varrho}^{\mathbb{E}}(0) \subset \mathbb{R}^3$ lies within an $\varepsilon$-neighborhood of $\mathbf{u}(B_{\varrho}(p))$, where $B_{\varrho}(p)$ is a geodesic ball in $(M,g)$.
\end{theorem}

\begin{proof}
First, set $r=3(C_{\mathrm{g}}+ 1)\varrho$ where $C_{\mathrm{g}}$ is the global gradient bound constant from Proposition~\ref{lem:L2est}. We will apply the previous lemma with this choice of $r$. Let $x=(x^1,x^2,x^3)\in B_{\varrho-\varepsilon/2}^{\mathbb{E}}(0)$. Our goal is to produce a point $w \in B_{\varrho}(p)$ such that $\left|\mathbf{u}(w)-x\right|<\varepsilon/2$, when $m(g)$ is sufficiently small.

The point $w$ will be found by applying Lemma~\ref{yui3} three times. Let $w_1:=p$ and apply Lemma~\ref{yui3} with $y=w_1$ and $t_0= x^1$ to find a point $w^*_1 \in B_\rho(w_1)$ such that
\begin{equation}
\left| \mathbf{u} (\psi^1_{x^1} (w^*_1))- \mathbf{u}(w^*_1) -  x^1 e_1 \right| < \varepsilon/16.
\end{equation}
Now define $w_2:=\psi^1_{x^1}(w^*_1)$ and apply Lemma~\ref{yui3} with $y= w_2$ and $t_0=x^2$ to find a point $w^*_2\in B_\rho(w_2)$ such that
\begin{equation}
\left| \mathbf{u} (\psi^2_{x^2} (w^*_2))- \mathbf{u}(w^*_2) -  x^2 e_2 \right| <  \varepsilon/16.
\end{equation}
Note that by \eqref{tttt} we have
\begin{equation}
d(p,w_2)=d(p,\psi_{x^1}^1(w^*_1))\leq d(p,w^*_1)+d(w^*_1,\psi_{x^1}^1(w^*_1))\leq \rho +C_{\mathrm{g}} x^1<\varrho+C_{\mathrm{g}}\varrho,
\end{equation}
and therefore the conditions of Lemma \ref{yui3} are satisfied here  since
\begin{equation}
d(p,y)+C_{\mathrm{g}} t_0 +\rho=d(p,w_2)+C_{\mathrm{g}} x^2 +\rho<2C_{\mathrm{g}} \varrho+2\varrho< r.
\end{equation}
Defining $w_3:=\psi^2_{x^2} (w^*_2)$,
we do this one more time with $y=w_3$ and $t_0 =x^3$ to find $w^*_3\in B_{\rho}(w_3)$ such that
\begin{equation}
\left| \mathbf{u} (\psi^3_{x^3} (w^*_3))- \mathbf{u}(w^*_3) -  x^3 e_3 \right| <  \varepsilon/16.
\end{equation}
Again, the assumptions of the previous lemma are valid since
\begin{equation}
d(p,w_3)\leq d(p,w^*_2)+d(w^*_2,\psi_{x^2}^2(w^*_2))
\leq d(p,w_2)+d(w_2,w^*_2)+C_{\mathrm{g}} x^2< 2\varrho+2C_{\mathrm{g}}\varrho,
\end{equation}
which implies that for this choice of $y$ and $t_0$,
\begin{equation}
d(p,y)+C_{\mathrm{g}} t_0 +\rho=d(p,w_3)+C_{\mathrm{g}} x^3 +\rho< 3(C_{\mathrm{g}} +1)\varrho=r.
\end{equation}

We now claim that the point $w:=\psi^3_{x^3} (w^*_3)$ has the property we desire if we choose $\rho$ appropriately. Simply by using the triangle inequality and breaking down $x=\sum_{i=1}^3 x^i e_i$, we can estimate
\begin{align}
\begin{split}
\left|\mathbf{u}(w) - x\right| &\leq  \sum_{i=1}^3 \left(\left|\mathbf{u}(\psi^i_{x^i} (w^*_i))- \mathbf{u}(w^*_i) -x^i e_i \right| + \left|\mathbf{u}(w^*_i)- \mathbf{u}(w_i)\right| \right)\\
&<  3\varepsilon/16 + 9 C_{\mathrm{g}} \rho,
\end{split}
\end{align}
which may be made smaller than $\varepsilon/4$ by choosing $\rho$ small enough. From our construction, we know that $w\in B_r(p)$. The only thing left to verify is that $w\in B_{\varrho}(p)$. For this, we use Theorem~\ref{t:mainGH} to see that for small enough mass, we have
\begin{align}
\begin{split}
d(w, p) &< \left|\mathbf{u}(w) - \mathbf{u}(p)\right| + \varepsilon/4 \\
&=|\mathbf{u}(w)| + \varepsilon/4 \\
&< |x| +  \varepsilon/4  +  \varepsilon/4 \\
&< (\varrho- \varepsilon/2) + \varepsilon/2 = \varrho.
\end{split}
\end{align}
\end{proof}

\begin{proof}[Proof of Theorem \ref{t:main}.]
Let $(M, g, p)\in \mathcal{M}(b,\tau,\bar{m},\kappa,\mathbf{p})$, and let $\varrho>0$. The statement of Theorem \ref{t:main} is equivalent to saying that for any $\varepsilon>0$, there exists $\delta>0$ such that if $m(g)<\delta$, then  $d_{GH}(B_{\varrho}(p),B_{\varrho}^{\mathbb{E}} (0))<\varepsilon$, where $d_{GH}$ is the Gromov-Hausdorff distance.
By Theorem~\ref{t:mainGH}, we can choose $\delta$ so that $d_{GH}\left(B_{\varrho}(p), \mathbf{u}\left(B_{\varrho}(p) \right) \right) <\varepsilon/2$ and also that  $\mathbf{u}\left(B_{\varrho}(p) \right) \subset  B_{\varrho+\varepsilon}^{\mathbb{E}} (0)$. By Theorem~\ref{surjectiveGH}, we can also choose $\delta$ small enough so that 
$d_{GH}\left(\mathbf{u}\left(B_{\varrho}(p)\right),B_{\varrho}^{\mathbb{E}}(0) \right) <\varepsilon/2$. The result then follows from the triangle inequality for $d_{GH}$.
\end{proof}

\section{Proof of Theorem \ref{t:main1}}
\label{sec7} \setcounter{equation}{0}
\setcounter{section}{7}

Consider the first part of Theorem \ref{t:main1}. For this we can follow the same line of argument used to prove Theorem \ref{t:main}. We need only identify which propositions and lemmas use the nonnegative scalar curvature assumption directly. It turns out that only Propositions \ref{prop:sup} and \ref{lem:L2est} use this hypothesis directly. All other  results rely on nonnegative scalar curvature solely through these two propositions.
Here is our modified version of Proposition~\ref{prop:sup}, which loosens the nonnegative scalar curvature assumption, but also requires a Ricci lower bound.

\begin{proposition}\label{prop:sup2}
Let $(M, g, p)$ be a pointed orientable complete $(b,\tau)$-asymptotically flat manifold with $H_2(M,\mathbb{Z})=0$, $\Ric(g)\ge -2\kappa g$, and $m(g) \leq\bar{m}$.  Further assume that $R_{g}\geq |X|^2 +\div X -\psi$, for some vector field $X\in C^1_{-2-\delta}(M)$ and some
nonnegative function $\psi\in L^1(M)$ supported in $M_{\hat{r}}$.  For sufficiently large  $r_0>0$ (depending on $\hat{r}$, $b$, and $\tau$) and $r_1>r_0$, there exist constants $C(\hat{r}, r_0, r_1, b,\tau,\bar{m}, \kappa)$ and $\hat{\varepsilon}(\hat{r}, r_0, r_1, b, \tau, \bar{m}, \kappa)>0$ such that if
$\| \psi  \|_{L^1( M)} \le \hat{\varepsilon}$,
then
\begin{equation}\label{e:supbound2}
\sup_{M_{r_1}}|u|\leq C(\hat{r}, r_0, r_1, b,\tau,\bar{m}, \kappa),
\end{equation}
where $u$ is an asymptotically linear harmonic function (as defined in Section~\ref{sec2}) whose average over $\mathcal{A}_{r_0,r_1}$ is zero.
\end{proposition}

\begin{proof}
Choose $r_0> \hat{r}$ large enough (depending on $b$, $\tau$) so that $g$ is uniformly equivalent to the Euclidean metric on $M\setminus M_{r_0}$, and let
$r_1>r_0$. We will show that the $X$ term can essentially be ignored, and  that the $\psi$ term has a negligible effect.
Within the mass inequality (Theorem~\ref{t:mass}) and the scalar curvature assumption, we may integrate the divergence term by parts to find
\begin{align}\label{jfhswitribh}
\begin{split}
16\pi m(g) &\geq\int_{M}\left(\frac{|\nabla^2 {u}|^2}{|\nabla {u}|}+\left(|X|^2 +\div X - \psi\right)|\nabla{u}|\right)dV\\
&\geq\int_{M}\left(\frac{|\nabla^2 {u}|^2}{|\nabla {u}|}
+|X|^2\cdot |\nabla{u}|  - |X|\cdot |\nabla^2 u| - \psi |\nabla u|   \right)dV\\
&\geq\int_{M}\left(\frac{|\nabla^2 {u}|^2}{2|\nabla {u}|}
+\frac{1}{2}|X|^2\cdot |\nabla{u}|  - \psi |\nabla u|   \right)dV\\
&\geq \int_{M}2 \left|\nabla \sqrt{|\nabla {u}|}\right|^2\, dV
-\left(\sup_{M_{\hat{r}}}|\nabla {u}|\right) \hat{\varepsilon},
\end{split}
\end{align}
where Young's inequality was used in the third line, and $\|\psi\|_{L^1(M)}\leq \hat{\varepsilon}$ together with \eqref{eq:straight} was used in the last line. Note that the hypothesis $X\in C^1_{-2-\delta}$ guarantees that all terms involving $X$ are integrable, and that the boundary term arising from integration by parts vanishes.

Following the next steps of the proof of Proposition \ref{prop:sup}, we see that in place of~\eqref{adfew} we obtain
\begin{equation}\label{adfew2}
\int_{\mathcal{A}_{r_0/2, 2r_1}}|\nabla {u}|^2 \,dV \leq C_1(r_0, r_1, b,\tau, \bar{m}) + C_2(r_0, r_1, b,\tau)\cdot
\left(\sup_{M_{\hat{r}}}|\nabla {u}|\right)^2 \hat{\varepsilon}^2.
\end{equation}
This may be combined with~\eqref{estHess},~\eqref{Poincare}, the Sobolev inequality for $C^0\subset H^2$, and the maximum principle for $u$, all of which remain unchanged, to produce
\begin{equation} \label{eq:supepsilon}
\sup_{M_{r_0}} |u| \le C_3(r_0, r_1, b,\tau, \bar{m}) + C_4(r_0, r_1, b,\tau)\cdot
\left(\sup_{M_{\hat{r}}}|\nabla {u}|\right) \hat{\varepsilon}.
\end{equation}
Next we use the Ricci lower bound in order to invoke the Cheng-Yau gradient estimate. In particular, choose $\rho(\hat{r}, r_0, b, \tau)$ small enough so that every $\rho$-ball centered at points $q\in M_{\hat{r}}$ is contained in $M_{r_0}$, then we can apply Theorem~\ref{t:CY} to the balls $B_{\rho/2}(q)\subset B_{\rho}(q)$ to get
\begin{equation}
\sup_{M_{\hat{r}}}|\nabla {u}| \le C_5(\hat{r}, r_0, b, \tau, \kappa)\cdot \sup_{M_{r_0}} |u|.
\end{equation}
This estimate shows that for sufficiently small $\hat{\varepsilon}$, the second term on the right side of~\eqref{eq:supepsilon} can be absorbed into the left side, proving the desired inequality.
\end{proof}

\begin{remark}
In view of the application of Young's inequality in the proof above, it is clear that the $|X_l|^2_{g_l}$ term in~\eqref{alaifiaw3} in the hypotheses of Theorem~\ref{t:main1} can be replaced by $c|X_l|^2_{g_l}$ for any constant $c>\frac{1}{4}$.
\end{remark}

In order to prove Theorem~\ref{t:main1}, for every lemma and proposition depending on Proposition~\ref{prop:sup2}, we will have to add hypotheses involving $\hat{r}$ and  $\hat{\varepsilon}$, but in addition to that, in view of~\eqref{jfhswitribh} and the proof of
Proposition \ref{lem:L2est}, we see that the biggest change is that we  have to replace~\eqref{eq:L2est} of Proposition \ref{lem:L2est}  by
\begin{equation}
\int_M|\nabla ^2 u^i|^2\,dV \leq C(\hat{r}, b,\tau,\bar{m},\kappa)\cdot \left[m(g) + \| \psi \|_{L^1(M)}\right].
\end{equation}
From there on it just means that every assumption of  smallness of $m(g)$ should be replaced by smallness of the sum $m(g) + \| \psi \|_{L^1(M)}$.
Of course, we will also need to alter the family $\mathcal{M}$ to include the parameter $\hat{\varepsilon}$ controlling  $\| \psi \|_{L^1(M)}$, chosen small enough so that  Proposition~\ref{prop:sup2} holds.


For the second part of Theorem \ref{t:main1}, we will show that the assumption that $R_{g_l}$ is bounded above by some constant $C_0$ implies that $H^2(M_l, \mathbb{Z})=0$ for sufficiently large $l$. Suppose that was not the case. Then there would be a subsequence with non-vanishing second homology, which implies the existence of an outermost minimal surface $\Sigma_l$ in $(M_l, g_l)$.  According to \cite[Lemma 4.1]{HI}, which does not require nonnegative scalar curvature, each component of $\Sigma$ is a 2-sphere.  If $K_l$, $A_l$, and $\nu_l$ denote the Gaussian curvature, second fundamental form, and unit normal of $\Sigma_l$, respectively, then  the Gauss-Bonnet and the traced Gauss equation tell us that
\begin{align}\label{eq:nominimal}
\begin{split}
4\pi &\le  \int_{\Sigma_l}K_l \,dA_l \\
&=  \int_{\Sigma_l}   \frac{1}{2} \left( R_{g_l}  -|A_l|^2-2\mathrm{Ric}_{g_l}(\nu_l,\nu_l)\right)dA_l \\
&\le |\Sigma_l| \left(\frac{1}{2}C_0 +2\kappa\right).
\end{split}
\end{align}
Since $m(g_l)\to0$, the Riemannian Penrose inequality~\cite{Bray} implies that $|\Sigma_l|\to 0$ as well. Therefore we have a contradiction, completing the proof of Theorem~\ref{t:main1}.

\begin{remark}
Careful examination of the proof above shows that all we really needed was to assume that $\| R_{g_l}^+\|_{L^p(\Sigma_l)}\le C_0$ for some $p\in(1,\infty]$.
\end{remark}



\end{document}